\documentclass[final,leqno,onefignum,onetabnum]{siamltex1213}

\usepackage{amssymb,amsmath,enumerate,empheq}
\usepackage{mathpazo}
\newtheorem{defn}{Definition}

\newtheorem{prop}{Proposition}
\newtheorem{rem}{Remark}
\newtheorem{assum}{Assumption}
\newtheorem{thm}{Theorem}
\newtheorem{ex}{Example}

\newtheorem{lem}{Lemma}

\newcommand{\norm}[1]{\Vert#1\Vert}

\newcommand{\abs}[1]{\left\vert#1\right\vert}
\newcommand{\set}[1]{\left\{#1\right\}}
\newcommand{\Real}{\mathbb R}
\newcommand{\eps}{\varepsilon}
\renewcommand{\phi}{\varphi}

\renewcommand{\d}{\mathbf{d}}
\newcommand{\dt}{\widetilde{\mathbf{d}}}

\renewcommand{\subset}{\subseteq}

\newcommand{\U}{\mathcal{U}}

\newcommand{\dom}{\text{dom}\,}

\newcommand{\F}{\mathcal{F}}
\newcommand{\G}{\mathcal{G}}
\newcommand{\M}{\mathcal{M}}
\newcommand{\Mbl}{\mathcal{M}_{b,\lambda}^{\Delta}}
\newcommand{\wMl}{\widehat{\mathcal{M}}^{\Delta}_{\lambda}}
\newcommand{\wB}{\widehat{\mathcal{B}}}
\newcommand{\bl}{b,\lambda}
\newcommand{\C}{\mathcal{C}}
\newcommand{\W}{\mathcal{W}}
\newcommand{\K}{\mathcal{K}}
\newcommand{\KL}{\mathcal{KL}}
\newcommand{\D}{\mathcal{D}}
\newcommand{\T}{\mathcal{T}}

\newcommand{\A}{\mathcal{A}}
\newcommand{\AD}{\mathcal{A}^{\Delta}}
\newcommand{\Atj}{\mathcal{A}_{[t,j]}^{\Delta}}
\newcommand{\HM}{\mathcal{H_M}}

\renewcommand{\H}{\mathcal{H}}

\newcommand{\gph}{\text{gph}\,}
\newcommand{\rge}{\text{rge}\,}
\newcommand{\ra}{\rightarrow}
\newcommand{\raa}{\rightrightarrows}
\newcommand{\rap}{\overset{\text{gph}}{\longrightarrow}}

\newcommand{\Z}{\mathbb{Z}}
\newcommand{\B}{\mathcal{B}}

\newcommand{\limg}{\text{lim}\,\text{gph}}
\newcommand{\ms}{\mathcal{M}_S}

\newcommand{\MD}{\mathcal{M}^{\Delta}}
\newcommand{\wMDl}{\widehat{\M}_{\lambda}^{\Delta}}
\newcommand{\HMD}{\mathcal{H}_{\mathcal{M}}^{\Delta}}
\newcommand{\SHMD}{\mathcal{S}_{\mathcal{H}_\mathcal{M}^{\Delta}}}

\title{Hybrid Systems with Memory: \\Existence and Well-posedness of Generalized Solutions\thanks{This work was supported, in part, by Royal Society grant IE130106, EU FP7 grant PCIG13-GA-2013-617377, NSERC grant RGPIN-2016-04139, US AFOSR grant FA9550-12-1-0127, US NSF grant ECCS-1232035, AFOSR FA9550-15-1-0155, and NSF ECCS-1508757.}} 

\author{Jun Liu\thanks{Department of Applied Mathematics, University of Waterloo, Waterloo, Ontario N2L 3G1, Canada (e-mail: j.liu@uwaterloo.ca).} \and Andrew R. Teel\thanks{Center for Control, Dynamical Systems, and Computation and Department of Electrical and Computer Engineering, University of California, Santa Barbara, CA 93106, USA (e-mail: teel@ece.ucsb.edu).}}

\begin{document}
\maketitle
\slugger{sicon}{xxxx}{xx}{x}{x--x}

\begin{abstract}
Hybrid systems with memory refer to dynamical systems exhibiting both hybrid and delay phenomena. While systems of this type are frequently encountered in many physical and engineering systems, particularly in control applications, various issues centered around the robustness of hybrid delay systems have not been adequately dealt with. In this paper, we establish some basic results on a framework that allows to study hybrid systems with memory through generalized concepts of solutions. In particular, we develop the basic existence of generalized solutions using regularity conditions on the hybrid data, which are formulated in a phase space of hybrid trajectories equipped with the graphical convergence topology. In contrast with the uniform convergence topology that has been often used, adopting the graphical convergence topology allows us to establish well-posedness of hybrid systems with memory. We then show that, as a consequence of well-posedness, pre-asymptotic stability of well-posed hybrid systems with memory is robust. 
\end{abstract}

\begin{keywords}
Hybrid systems, time delay, functional inclusions, generalized solutions, basic existence, well-posedness, robust stability.
\end{keywords}

\begin{AMS}34K34, 34A38, 93D09\end{AMS}

\pagestyle{myheadings}
\thispagestyle{plain}
\markboth{J.~Liu and A.~R.~Teel}{Hybrid Systems with Memory}

\section{Introduction}

Hybrid systems with memory refer to dynamical systems exhibiting both hybrid and delay phenomena. Control systems with delayed hybrid feedback and interconnected hybrid systems with network delays are particular examples of such systems. In fact, delays are often inevitable in many control applications \cite{sipahi2011stability} and often cause instability and/or loss of robustness \cite{cloosterman2009stability}.

Motivated by robust stability issues in hybrid feedback control systems, generalized solutions of hybrid inclusions, defined on hybrid time domains, have been proposed and proven effective for hybrid systems without memory  \cite{goebel2004hybrid,goebel2012hybrid,sanfelice2008generalized}. This has led to most of the stability analysis tools and results for classical nonlinear systems, including converse Lyapunov theorems, being successfully extended to hybrid systems (see \cite{goebel2012hybrid,sanfelice2007invariance} and references therein).

The main motivation for the research presented in this paper is to provide a theoretical framework for studying the robust stability of hybrid systems with memory via generalized solutions. An initial step was taken in \cite{liu2012generalized}, which considered a phase space that is equipped with a suitable notion of convergence, namely the graphical convergence topology. Using tools from functional differential inclusions, some basic existence and nominal well-posedness results were established (though detailed proofs were not included). The case considered in \cite{liu2012generalized} assumes that the flow and jump sets are subsets of the Euclidean space. This leaves open the general (and more challenging) case, where the flow and jump sets are subsets of the space of hybrid memory arcs. The more recent paper \cite{liu2014hybrid-ifac} considers this general case. Following these preliminary results, the current paper provides a detailed account of this new development.

While asymptotic stability for hybrid systems with delays has been addressed in the past in various settings (see, e.g., \cite{chen2009input,liu2001uniform,liu2006stability,liu2011generalized,liu2011input,sun2012integral,yan2008stability,yuan2003uniform,zhang2014global}), general results on robust asymptotic stability along the lines of \cite{goebel2012hybrid} for hybrid systems with delays are still not available. This is partially owing to the fact that most current tools and results rely on standard concepts like uniform convergence, while this concept is not well-suited to handle discontinuities caused by jumps in hybrid systems, especially when structural properties of the solutions are concerned. It is from this perspective that we believe it is necessary to formulate hybrid systems with memory using generalized concepts of solutions. 

The main contributions of this paper are twofold, {\color{black} both of which are motivated by stability and control of hybrid systems with delays.} First, we prove the basic existence results for hybrid systems with memory with general hybrid data. While these results extend earlier results in \cite{haddad1981monotone} on functional differential inclusions to hybrid functional inclusions, {\color{black} they are also fundamental to the recent development in stability theory for hybrid systems with delays \cite{liu2016lyapunov,liu2016invariance,liu2016razumikhin}.} Second, {\color{black} motivated by the importance of robust stability in control theory,} {\color{black} we investigate hybrid systems with memory under perturbations. More specifically,} we formulate perturbations of hybrid data for hybrid systems with memory and establish a well-posedness result for hybrid systems with memory satisfying the basic assumptions. As an immediate consequence of this well-posedness result, it is proved that pre-asymptotic stability for well-posed hybrid systems with memory is robust.

The rest of the paper is organized as follows. Preliminaries on hybrid systems data, set-valued analysis, the phase space of hybrid solutions, and regularity assumptions on hybrid data are presented in Section \ref{sec:pre}. A general basic existence result is stated and proved in Section \ref{sec:exist}. Perturbations of hybrid data and well-posedness are defined and proved in Section \ref{sec:well-posed}. As a consequence of well-posedness, results on the robustness of pre-asymptotic stability are presented in Section \ref{sec:robustness}.

\section{Preliminaries}\label{sec:pre}

\subsubsection*{Notation:} $\Real^n$ denotes the $n$-dimensional Euclidean space with its norm denoted by  $\abs{\cdot}$; $\mathbb{Z}$ denotes the set of all integers; $\Real_{\ge 0}= [0,\infty)$, $\Real_{\le 0}=(-\infty,0]$, $\Z_{\ge 0}=\set{0,\,1,\,2,\,\cdots}$, and $\Z_{\le 0}=\set{0,\,-1,\,-2,\,\cdots}$; $C([a,b],\Real^n)$ denotes the set of all continuous functions from $[a,b]$ to $\Real^n$.

\subsection{Hybrid systems with memory}

We start with the definition of hybrid time domains and hybrid arcs
\cite{goebel2012hybrid,goebel2006solutions} for hybrid systems and generalize them in order to define hybrid systems with memory.

\begin{defn}\em \label{def:hybridarc}
{\color{black} Consider a subset $E\subset \Real\times \Z$ with $E=E_{\ge 0}\cup E_{\le 0}$, where $E_{\ge 0}:=(\Real_{\ge 0}\times \Z_{\ge 0})\cap E$ and $E_{\le 0}:=(\Real_{\le 0}\times \Z_{\le 0})\cap E$. The set $E$ is called a \emph{compact hybrid time domain with memory} if}
$$
E_{\ge 0}=\bigcup_{j=0}^{J-1}([t_j,t_{j+1}],j)
$$
and
$$
E_{\le 0}=\bigcup_{k=1}^{K}([s_k,s_{k-1}],-k+1)
$$
for some finite sequence of times $s_{K}\le\cdots\le s_1\le s_0=0=t_0\le t_1\le \cdots\le t_J$. It is called a \emph{hybrid time domain with memory} if, for all $(T,J)\in E_{\ge 0}$ and all $(S,K)\in \Real_{\ge 0}\times \Z_{\ge 0}$,
$
(E_{\ge 0}\cap ([0,T]\times\set{0,\,1,\,\cdots,\,J})) \cup (E_{\le 0}\cap ([-S,0]\times\set{-K,\,-K+1,\,\cdots,\,0}))
$
is a compact hybrid time domain with memory. The set $E_{\le 0}$ is called a \emph{hybrid memory domain}.
\end{defn}

\begin{defn}\em \label{def:memoryarc}
A \emph{hybrid arc with memory} consists of a hybrid time domain with memory, denoted by $\dom x$, and a function $x:\,\dom x\ra \Real^n$ such that $x(\cdot,j)$ is locally absolutely continuous on $I^{j}=\set{t:\,(t,j)\in dom\,x}$ for each $j\in\Z$ such that $I^{j}$ has nonempty interior. In particular, a hybrid arc $x$ with memory is called a \emph{hybrid memory arc} if its domain is a hybrid memory domain.  
We shall simply use the term \emph{hybrid arc} if we do not have to distinguish between the above two hybrid arcs. We write $\dom_{\ge 0}(x):=\dom x\cap(\Real_{\ge 0}\times \Z_{\ge 0})$ and  $\dom_{\le 0}(x):=\dom x\cap(\Real_{\le 0}\times \Z_{\le 0})$.
\end{defn}

{\color{black} Fig. \ref{fig:memory} shows the graph of a hybrid memory arc and its domain. As we shall see in what follows, hybrid memory arcs play the role of initial data in hybrid systems with memory. They are essentially a collection of pieces of continuous functions defined on closed intervals. This is different from the classical formulation, where initial data for hybrid systems with delays (e.g., impulsive delay differential equations \cite{ballinger1999existence}) are taken to be piecewise continuous functions.}

\begin{figure}[ht!]
\centering
    \includegraphics[width=0.95\textwidth]{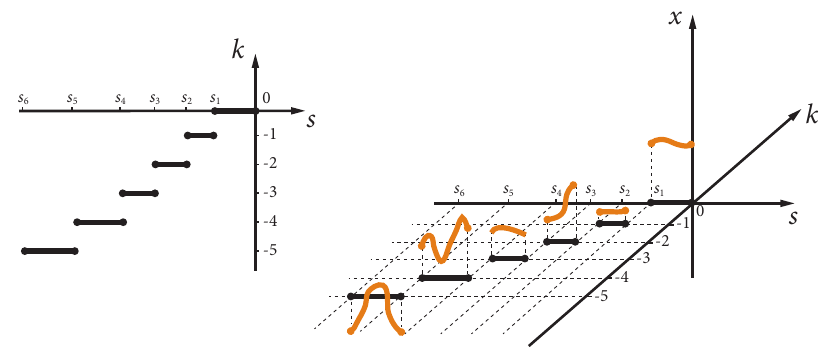}
  \caption{A typical hybrid memory arc (right) depicted together with its hybrid memory domain (left).}\label{fig:memory}
\end{figure}

{\color{black} In this paper, we consider generalized solutions to hybrid systems with memory given by hybrid arcs. Fig. \ref{fig:hybrid} shows a hybrid arc with its domain. A hybrid arc has a memory part as its initial data. The formal definitions of hybrid systems with memory and their solutions are given next.}

\begin{figure}[ht!]
\centering
    \includegraphics[width=0.95\textwidth]{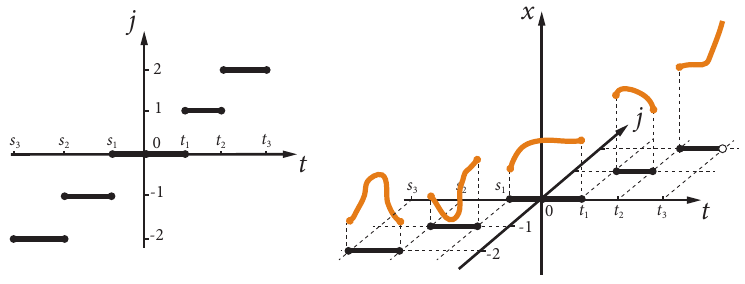}
  \caption{A typical hybrid arc with memory (right) depicted together with its hybrid time domain with memory (left).}\label{fig:hybrid}
\end{figure}

We shall use $\M$ to denote the collection of all hybrid memory arcs. Moreover, given $\Delta\in[0,\infty)$, we denote by $\MD$ the collection of hybrid memory arcs $\phi$ satisfying the following two conditions: (1) $s+k\ge -\Delta-1$ for all $(s,k)\in\dom\phi$; and (2) there exists $(s',k')\in \dom\phi$ such that $s'+k'\le-\Delta$. The constant $\Delta$ roughly captures the size of the memory for the system (see Remark \ref{rem:size} below). Given a hybrid arc $x$, we define an operator $\AD_{[\cdot,\cdot]}x:\,\dom_{\ge 0}(x)\ra \M$ by
\begin{equation}\label{eq:opA}
\AD_{[t,j]}x(s,k)=x(t+s,j+k),
\end{equation}
for all $(s,k)\in\dom (\AD_{[t,j]}x)$, where
\begin{equation}\label{eq:domA}
\dom (\AD_{[t,j]} x):=\Big\{(s,k)\in \Real_{\le 0}\times \Z_{\le 0}:\,(t+s,j+k)\in\dom x,\;s+k\ge -\Delta_{\inf}\Big\},
\end{equation}
where
\begin{equation}\label{eq:deltainf}
\Delta_{\inf}:=\inf\Big\{\delta\ge\Delta:\,\exists(s,k)\;\text{s.t.}\;(t+s,j+k)\in\dom x\;\text{and}\;   s+k=-\delta\Big\}.
\end{equation}
{\color{black} The following result follows immediately from the above definition.  
\begin{prop}
If $\AD_{[0,0]}x\in \MD$, then $\AD_{[t,j]}x\in \MD$ for any $(t,j)\in \dom_{\ge 0}(x)$.
\end{prop}
\begin{proof}
From Definitions \ref{def:hybridarc} and \ref{def:memoryarc} and the fact that $\AD_{[0,0]}x\in \MD$, we can show that, for any $(t,j)\in \dom_{\ge 0}(x)$, there exists $(s,k)$ such that 
\begin{equation}\label{eq:realization}
(t+s,j+k)\in\dom x,\quad s+k=-\delta,\quad\text{and}\quad\Delta\le \delta<\Delta+1. 
\end{equation}
By (\ref{eq:deltainf}) and (\ref{eq:realization}), we have $\Delta\le\Delta_{\inf}<\Delta+1$. This leads to the following:
\begin{itemize}
\item[(i)] By (\ref{eq:domA}), $s+k\ge -\Delta_{\inf}>-\Delta-1$ for all $(s,k)\in\dom(\AD_{[t,j]}x)$. 
\item[(ii)] Fix some $(s,k)$ and $\delta$ such that (\ref{eq:realization}) holds, then $\delta\ge \Delta_{\inf}$. By (\ref{eq:domA}) and (\ref{eq:realization}), $(s,k)\in \dom(\AD_{[t,j]}x)$ and $s+k=-\delta\le -\Delta$. 
\end{itemize}
This shows $\AD_{[t,j]}x\in \MD$. 
\end{proof}

}

\begin{rem}\label{rem:size}\em 
The constant $\Delta$ plays the role of distinguishing hybrid systems with finite or infinite memory. It is noted that the definition of $\MD$ is slightly unintuitive, as one would expect to it to include hybrid memory arcs with domain sizes exactly up to $\Delta$. We relax this to allow the size of memory arcs in $\MD$ to vary from $\Delta$ to $\Delta+1$. {\color{black} This relaxation allows us to prove the graphical convergence results needed to establish well-posedness and robustness of hybrid systems with memory later in Section \ref{sec:well-posed}.} 
\end{rem}

\begin{defn}\em 
A \emph{hybrid system with memory of size $\Delta$} is defined by a 4-tuple $\HMD=(\C,\F,\D,\G)$:
\begin{itemize}
\item a set $\C\subset\MD$, called the \emph{flow set};
\item a set-valued functional $\F:\MD\raa \Real^n$, called the \emph{flow map};
\item a set $\D\subset\MD$, called the \emph{jump set};
\item a set-valued functional $\G:\MD\raa \Real^n$, called the \emph{jump map}.
\end{itemize}
\end{defn}

{\color{black} Given a hybrid memory arc $\phi\in\MD$ and $g\in\Real^n$, we define $\phi_g^+$ to be a hybrid memory arc in $\MD$ satisfying $\phi_g^+(0,0)  = g$ and $\phi_g^+(s,k-1) = \phi(s,k)$ for all $(s,k)\in\dom \phi$. Furthermore, we define $\G^+(\D):=\set{\phi_g^+:\,g\in\G(\phi),\;\phi\in \D}$. Intuitively, $\phi_g^+$ is the hybrid memory arc following $\phi$ after taking a jump of value $g$; $\G^+(\D)$ is the set of hybrid memory arcs that can be obtained by applying the functional $\G$ on the jump set $\D$.}

\begin{defn}\label{def:sol}\em
 A hybrid arc is a \emph{solution to the hybrid system $\HMD$} if $\AD_{[0,0]}x\in \C\cup \D$ and:
 \begin{enumerate}[(S1)]
 \item for all $j\in \Z_{\ge 0}$ and almost all $t$ such that $(t,j)\in\dom x$,
 \begin{equation}
 \Atj x\in \C,\quad \dot{x}(t,j)\in \F(\Atj x),
 \end{equation}
 \item for all $(t,j)\in\dom x$ such that $(t,j+1)\in\dom x$,
 \begin{equation}
 \Atj x\in \D,\quad x(t,j+1)\in \G(\Atj x).
 \end{equation}
 \end{enumerate}
The solution $x$ is called \emph{nontrivial} if $\dom_{\ge 0}(x)$ has at least two points. It is called \emph{complete} if $\dom_{\ge 0}(x)$ is unbounded. It is called \emph{maximal} if there does not exist another solution $y$ to $\HMD$ such that $\dom x$ is a proper subset of $\dom y$ and $x(t,j)=y(t,j)$ for all $(t,j)\in\dom x$. The set of all maximal solutions to $\HMD$ is denoted by $\SHMD$.
\end{defn}

{\color{black} 

We first use a simple delay differential equation subject to periodic state jumps to illustrate the concepts introduced by the previous definitions. 

\begin{ex}\em 
Consider the scalar delay differential equation 
$$
\dot{x}(t)=x(t-2),\quad t\not\in\Z_{\ge 0},
$$
subject to state jumps defined by 
$$
x(t)=2x(t^{-}),\quad t\in\Z_{\ge 0},
$$
where $x(t^{-})=\lim_{s\ra t^{-}}x(s)$. 
\end{ex}
Normally, to study an impulsive delay differential equation like this, one needs to consider the initial data of the equation to be a piecewise continuous function defined on the interval $[-2,0]$ (see, e.g., \cite{ballinger1999existence}). Here we formulate this system as a hybrid system with memory as follows. 

Consider $\MD$ with $\Delta=3$. Let $\psi=(\phi,\tau)\in\MD$, where both $\phi$ and $\tau$ are 1-dimensional hybrid arcs with shared domains. The variable $\tau$ plays the role of a timer, as seen in hybrid systems without memory \cite{goebel2012hybrid}. The above system can be interpreted as a hybrid system with memory in $\MD$ with the following data:
$$
\F(\psi)=\begin{bmatrix}\phi(-2,k))\\
1
\end{bmatrix},\quad k=\max\set{j:\,(-2,j)\in\dom\psi},\quad \psi\in \MD,
$$
$$
\G(\psi)=\begin{bmatrix} 2\phi(-2,0))\\
0
\end{bmatrix},\quad\psi\in \MD,
$$
$$
\C=\set{\psi=(\phi,\tau)\in\MD:\,\tau(0,0)\subset[0,\delta]}, 
$$
$$
\D=\set{\psi=(\phi,\tau)\in\MD:\,\tau(0,0)=\delta}.
$$

{\color{black}

The following example motivates the need to consider hybrid systems with general flow and jump sets that are subsets of $\M$, compared with earlier work in \cite{liu2012generalized}, where the flow and jump sets are subsets of $\Real^n$.  

\begin{ex}\em \label{ex:ec}
Consider an event-triggered control system \cite{postoyan2015framework} with the plant dynamics given by
\begin{equation}
\dot{x}=f(x,u),
\end{equation}
where $x\in\Real^n$ is the plant state and $u\in\Real^m$ is the control input, for which a stabilizing static state-feedback controller is designed as
\begin{equation}
u = \alpha (x),
\end{equation}
where $\alpha$ is some continuous function from $\Real^n$ to $\Real^m$. Suppose that a zero-order-hold device is used to implement the controller, which is connected to a sensor measuring  the state of the plant through a network. Between updates, the control input is kept constant and, during updates, the control input is set to be equal to $\alpha(x_s)$, where $x_s$ is the last sampled state. The event-triggering law for updating the control input is defined by 
\begin{equation}
|u-\alpha(x_s)|\geq \rho(x_s),
\end{equation}
where $\rho:\Real^n\ra \Real_{\ge 0}$ is a positive definite function. We assume that there is a possibly time-varying sampling delay from the sensor to controller denoted by $h_s$, whereas the input delay from the controller to the plant is denoted by $h_u$. The overall control scheme is illustrated in Fig. \ref{fig:ec}. 

\begin{figure}[ht!]
\centering
    \includegraphics[width=0.5\textwidth]{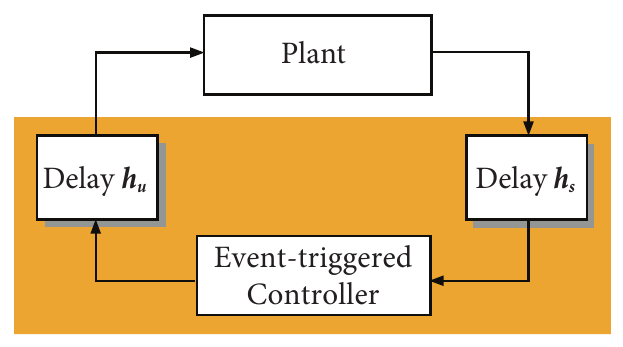}
  \caption{Illustration of an event-triggered control system with delays.}\label{fig:ec}
\end{figure}

Suppose that $h_s+h_u\le h$ for some constant $h>0$. Consider $\Delta=h+1$. Let $\textbf{z}=(\phi,\psi)\in\MD$, where $\phi$ and $\psi$ are $n$-dimensional and $m$-dimensional hybrid arcs, respectively, satisfying $\dom \phi=\dom\psi=\dom \textbf{z}$. From the perspective of the controller, the event-triggered control system above can be modelled as a hybrid system with memory in $\MD$ with the following data:
$$
\F(\textbf{z})=\begin{bmatrix}f(\phi(0,0),\psi(-h_u,k_u))\\
0
\end{bmatrix},\quad k_u=\max\set{k:\,(-h_u,k_u)\in\dom \textbf{z}},\quad \textbf{z}\in \MD,
$$
$$
\G(\textbf{z})=\begin{bmatrix} \phi(0,0))\\
\alpha(\phi(-h_s,k_s))
\end{bmatrix},\quad k_s=\max\set{k:\,(-h_s,k_s)\in\dom \textbf{z}},\quad \textbf{z}\in \MD,
$$
$$
\C=\set{\textbf{z}=(\phi,\psi)\in\MD:\,\abs{\psi(0,0)-\alpha(\phi(-h_s,k_s))}\le \rho(\phi(-h_s,k_s))},
$$
$$
\D=\set{\textbf{z}=(\phi,\psi)\in\MD:\,\abs{\psi(0,0)-\alpha(\phi(-h_s,k_s))}\ge \rho(\phi(-h_s,k_s))}.
$$
\end{ex}

Note that it is necessary to define the flow and jump sets in $\MD$ because of the state-dependent triggering law involving delayed states. 

}

\begin{rem}\em
While it seems more cumbersome to formulate hybrid systems with memory as hybrid inclusions, this general formulation does allow us to investigate asymptotic properties that cannot be conveniently studied using classical notions of solutions, e.g., invariance principles \cite{liu2016invariance} and robust stability (to be discussed later in Section \ref{sec:robustness}). {\color{black} For instance, it is recently demonstrated in \cite{postoyan2015framework} that the hybrid system theory without memory \cite{goebel2012hybrid} provides a unified framework for analyzing the stability of event-trigged control systems without delays. Since delays are inevitable due to the use of networked control, the theory to be developed in this paper, together with the stability analysis results in \cite{liu2016lyapunov,liu2016invariance,liu2016razumikhin}, can play a useful role in analyzing the type of systems described in Example  \ref{ex:ec}. This is the main motivation for this work.}
\end{rem}

}

\subsection{Preliminaries on set-valued analysis}

In this section, we recall a few concepts from set-valued analysis \cite[Chapters 4 \& 5]{rockafellar1998variational} (see also Chapter 5 of \cite{goebel2012hybrid}).

\begin{defn}[Set convergence]\em
Consider a sequence of sets $\set{H_i}_{i=1}^{\infty}$ in $\Real^n$. The \emph{outer limit} of the sequence, denoted by $\limsup_{i\ra\infty} H_i$ is the set of all $x\in\Real^n$ for which there exists a subsequence $x_{i_k}\in S_{i_k}$, $k=1,2,\cdots$, such that $x_{i_k}\ra x$. The \emph{inner limit} of $\set{H_i}_{i=1}^{\infty}$, denoted by $\liminf_{i\ra\infty} H_i$, is the set of all $x\in\Real^n$ for which there exists a sequence $x_{i}\in S_{i}$ such that $x_{i}\ra x$. The \emph{limit} of $\set{H_i}_{i=1}^{\infty}$ exists if $\limsup_{i\ra\infty} H_i=\liminf_{i\ra\infty} H_i$ and it is then given by $\lim_{i\ra\infty} H_i=\limsup_{i\ra\infty} H_i=\liminf_{i\ra\infty} H_i$.
\end{defn}

\begin{defn}[Set-valued mappings]\em
Let $S:\,\Real^m\raa\Real^n$ be a set-valued mapping from $\Real^m$ to $\Real^n$. Its domain, range, and graph are defined, respectively, by 
\begin{align*}
\dom S&:=\set{x:\,S(x)\neq\emptyset},\\
\rge S&:=\set{y:\,\exists\;x\;\text{s.t.}\; y\in S(x)},\\
\gph S&:=\set{(x,y):\,y\in S(x)}.
\end{align*}
\end{defn}

\begin{defn}[Graphical convergence]\em 
A sequence of set-valued mappings $S_i:\,\Real^m\raa\Real^n$ is said to \emph{converge graphically} to some $S:\,\Real^m\raa\Real^n$ if $\lim_{i\ra\infty}\gph S_i=\gph S$. We use $\rap$
to denote graphical convergence.
\end{defn}

While a hybrid arc $\phi$ is a single-valued map on its domain, it can also be seen as a set-valued mapping from $\Real^2$ to $\Real^n$, with its values defined by $\phi(s,k)$ if $(s,k)\in\dom\phi$ and $\emptyset$ otherwise. We say that a sequence of hybrid arcs $\phi_i:\,\dom\phi_i\ra\Real^n$ converges graphically to some set-valued mapping $\phi:\,\Real^2\raa\Real^n$ if $\lim_{i\ra\infty}\gph \phi_i=\gph \phi$. Note that the graphical limit of a sequence of hybrid arcs can be set-valued and in general may not be an hybrid arc.

\subsection{The space $(\MD,\mathbf{d})$}

It is clear that $\MD$ is not a vector space, since different hybrid arcs can have different domains. In this section, we recall from \cite{rockafellar1998variational} a quantity that characterizes the set convergence of closed nonempty sets and use this distance to define a metric on $\MD$. Let $\text{cl-sets}_{\not\equiv\emptyset}(\Real^n)$ denote the collection of all nonempty, closed subsets of $\Real^n$.

Given $\rho\ge 0$, for each pair $A,B\in \text{cl-sets}_{\not\equiv\emptyset}(\Real^n)$, define
$$
\mathbf{d}_\rho(A,B):=\max_{\abs{z}\le \rho}\big\vert d(z,A)- d(z,B)\big\vert.
$$
where $d(z,H)$ for $z\in\Real^{n}$ and $H\subset\Real^{n}$ is defined by $\inf_{w\in H}\abs{w-z}$. Furthermore, define
$$
\mathbf{d}(A,B):=\int_0^\infty \mathbf{d}_\rho(A,B) e^{-\rho}d\rho,
$$
which is called the \emph{(integrated) set distance} between $A$ and $B$. This distance indeed characterizes set convergence of sets in $\text{cl-sets}_{\not\equiv\emptyset}(\Real^n)$ as recalled below.

\begin{thm}\cite[Theorem 4.42]{rockafellar1998variational}\label{thm:distance}
A sequence $S_i\in \text{cl-sets}_{\not\equiv\emptyset}(\Real^n)$ converges to $S$ if and only $\mathbf{d}(S_i,S)\ra 0$. Moreover, the space $(\text{cl-sets}_{\not\equiv\emptyset}(\Real^n),\mathbf{d})$ is a separable, locally compact, and complete metric space.
\end{thm}

We apply this distance on graphs of hybrid arcs as follows. Given a hybrid arc $\phi:\,\dom \phi\ra\Real^n$, the graph of $\phi$ is defined by
$
\gph \phi:=\set{(s,k,x):\,x=\phi(s,k)}.
$
Given $\rho\ge 0$, for a pair of hybrid arcs $\phi$ and $\psi$, define
$
\mathbf{d}_\rho(\phi,\psi):=\d_{\rho}(\gph \phi,\gph \psi)
$
and
$
\mathbf{d}(\phi,\psi):=\d(\gph \phi,\gph \psi),
$
which is called the \emph{graphical distance} between hybrid arcs. Note that the same notion of graphical distance applies to both hybrid arcs and hybrid memory arcs.

We now focus on hybrid memory arcs in $\MD$. {\color{black} We first note that the graph of a hybrid memory arc is a nonempty, closed subset. Indeed, it is nonempty because the domain of a hybrid memory arc has at least one point $(0,0)$. It is closed because, by definition, the domain of a hybrid memory arc is the union of a finite number of closed intervals and a hybrid memory arc is continuous in its first argument.} As a consequence of Theorem \ref{thm:distance} above and the fact that the graph of a hybrid memory arc is a nonempty, closed subset of $\Real^{n+2}$, we know that the space $(\MD,\mathbf{d})$ is itself a metric space. However, $(\MD,\d)$ is not complete, since the limit of a sequence of graphically convergent hybrid memory arcs may not be a hybrid memory arc, {\color{black} as shown in the following example.}

{\color{black} \begin{ex}
Consider a sequence of hybrid memory arcs  $\set{\phi_n}_{n=1}^{\infty}$  in $\M^{\Delta}$ (with $\Delta=1$) defined by $\phi_n(s,0)= 0$ for $s\in [-1,-\frac{1}{n}]$ and $\phi_n(s,0) = \frac{ns+1}{1+n}$ for $s\in [-\frac{1}{n},0]$. As shown in Fig. \ref{fig:example}, it is clear that the graphs of these hybrid memory arcs converge to the set 
$$\set{(s,k,x):\,-1\le s\le 0,\,k=0,\,x=0}\cup \set{(s,k,x):\, s=0,\,k=0,\,0\le x\le 1}.$$
\begin{figure}[ht!]
\centering
    \includegraphics[width=0.6\textwidth]{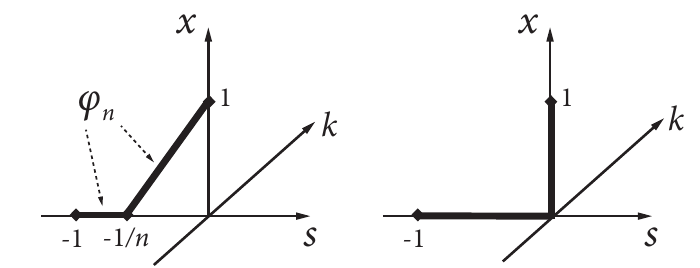}
  \caption{Hybrid memory arcs that graphically converge to a non-hybrid memory arc.}\label{fig:example}
\end{figure}
Note that this is a closed set that does not correspond to the graph of any hybrid memory arcs, because otherwise at the point $(0,0)$, the hybrid arc would not be single-valued. 
\end{ex}}

The following subspace of $(\MD,\d)$ is often used where such compactness properties are needed. Given $b,\lambda\in\Real_{\ge 0}$, define
\begin{align*}
\MD_{b}:&=\Big\{\phi\in\MD:\, \norm{\phi}:=\sup_{(s,k)\in\dom\phi}\abs{\phi(s,k)}\le b\Big\},\\
\MD_{b,\lambda}:&=\Big\{\phi\in\MD_{b}:\, \phi\text{ is }\lambda\text{-Lipschitz}\Big\},
\end{align*}
where $\phi\in\M$ is said to be \emph{$\lambda$-Lipschitz} if
$\abs{\phi(s,k)-\psi(s',k)}\le \lambda\abs{s-s'}$ holds for all $(s,k),\,(s',k)\in\dom\phi$.

\begin{prop}\label{thm:compact}
The space $(\MD_{b,\lambda},\mathbf{d})$ is a 
locally compact complete metric space.
\end{prop}

\begin{proof}
{ It suffices to show that $\MD_{b,\lambda}$ is a closed subspace of $(\text{cl-sets}_{\not\equiv\emptyset}(\Real^{n+2}),\mathbf{d})$ under the graphical distance. Consider a sequence $\phi_i\in \MD_{b,\lambda}$ such that $\d(\phi_i,\phi)\ra 0$ as $i\ra\infty$ for some $\phi\in(\text{cl-sets}_{\not\equiv\emptyset}(\Real^{n+2}),\mathbf{d})$. We need to prove that
$\phi \in \MD_{b,\lambda}$. Note that the sequence $\set{\phi_i}_{i=1}^\infty$ is a bounded sequence. 
It follows from the argument in \cite[Examples 5.3 and 5.19]{goebel2012hybrid} that
$
\dom\phi=\lim_{i\ra\infty}\dom\phi_i
$
is a hybrid memory domain. Moreover, since for each $(s,k)\in\dom\phi$, there exist $(s_i,k_i)\in\dom\phi_i$ such that $(s_i,k_i)\ra (s,k)$ as $i\ra \infty$. It follows that $s+k\ge -\Delta-1$, since $s_i+k_i\ge -\Delta-1$ for all $i$.

Now for each $k\in\Z_{\le 0}$, let $I^k=\set{s\in\Real_{\le 0}:\,(s,k)\in\dom\phi}$. Let $I_i^k$ be similarly defined for $\phi_i$. It follows from the very definition of set convergence that $\phi_i(\cdot,k)$ converges graphically to $\phi(\cdot,k)$. Now note that the sequence $\set{\phi_i(\cdot,k)}_{i=1}^\infty$ is $\lambda$-Lipschitz. Suppose $I^k$ is a nonempty set. Following the same argument as in the proof of \cite[Lemma 5.28]{goebel2012hybrid}, one can show that $\phi(\cdot,k)$ is single-valued and $\lambda$-Lipschitz on $I^k$. In addition, $\phi_i(\cdot,k)$ converges uniformly to $\phi(\cdot,k)$ on every compact subset of $\text{int}(I_i^k)$.
This concludes that $\phi\in \MD_{b,\lambda}$.}
\end{proof}

\subsection{Other measures of distance in $\MD$}\label{app:distance}

While the graphical distance $\mathbf{d}$ fully characterizes graphical convergence in $\MD$, in some cases it is not convenient to directly use it. We introduce some other quantities that can be used in company with $\mathbf{d}$ to characterize closeness of two hybrid memory arcs.

\subsection*{Uniform distance}
Given $\phi,\,\psi\in\M$ such that $\dom\phi=\dom\psi$, we can define $$\norm{\phi-\psi}=\sup_{(t,j)\in \dom\phi}\abs{\phi(t,j)-\psi(t,j)},$$ which is called the \emph{uniform distance} between $\phi$ and $\psi$.

\subsection*{$(\tau,\eps)$-closeness}
For hybrid arcs with possibly different domains, a notion called $(\tau,\eps)$-closeness  \cite{goebel2006solutions} can be used to measure their closeness. Here we modify the notion slightly to use it on hybrid memory arcs. Given $\phi,\,\psi\in\M$ and $\tau,\,\eps>0$, $\phi$ and $\psi$ are said to be \emph{$(\tau,\eps)$-close} if
\begin{itemize}
\item[(a)] for all $(t,j)\in\dom\phi$ with $\abs{t+j}\le\tau$, there exists some $s$ such that $(s,j)\in\dom\psi$, $\abs{t-s}\le \eps$, and $\abs{\phi(t,j)-\psi(s,j)}\le\eps$;
\item[(b)] for all $(t,j)\in\dom\psi$ with $\abs{t+j}\le\tau$, there exists some $s$ such that $(s,j)\in\dom\phi$, $\abs{t-s}\le \eps$, and $\abs{\phi(t,j)-\psi(s,j)}\le\eps$.
\end{itemize}
We write $\widetilde{\d}_{\tau}(\phi,\psi)\le \eps$ to indicate that $\phi$ and $\psi$ are $(\tau,\eps)$-close. If $\widetilde{\d}_{\tau}(\phi,\psi)\le \eps$ for all $\tau\ge 0$, $\phi$ and $\psi$ are said to be \emph{$\eps$-close} and we write $\widetilde{\d}(\phi,\psi)\le \eps$.

\subsection*{$(\rho,\eps)$-closeness of graphs}
More generally, we can also use the following to characterize the closeness of the graphs of two hybrid arcs $\phi,\,\psi\in\M$: there exists $\rho>0$ and $\eps>0$ such that $\gph\phi\cap\rho\B\subset\gph\psi+\eps\B$ and $\gph\psi\cap\rho\B\subset\gph\phi+\eps\B$. If the above holds, we say that the \emph{graphs} of $\phi$ and $\psi$ are \emph{$(\rho,\eps)$-close}. We write
$\widehat{\d}_{\rho}(\phi,\psi)\le \eps$
to indicate that the graphs of $\phi$ and $\psi$ are $(\rho,\eps)$-close.

A technical proposition that relates these different measures of distance in $\MD$ is proved in Appendix \ref{append:distance}. 

\subsection{Regularity assumptions on hybrid data of $\HMD$}\label{regfg}

We now introduce a few regularity conditions on the hybrid data of $\HMD=(\C,\F,\D,\G)$, which will allow us to establish certain basic existence and well-posedness results in the next section.

\begin{defn}[Outer semicontinuous]
A set-valued functional $\F:\MD\raa \Real^n$ is said to be \emph{outer semicontinuous} at $\phi\in\MD$ if, for every sequences of hybrid memory arcs $\phi_i\rap \phi$ and $y_i\ra y$ with $y_i\in \F(\phi_i)$, we have $y\in \F(\phi)$.
\end{defn}

\begin{defn}[Local boundedness]
A set-valued functional $\F:\MD\raa \Real^n$ is said to be \emph{locally bounded} at $\phi\in\MD$ if there exists a neighborhood $\U_{\phi}$ of $\phi$ such that the set $\F(\U_{\phi}):=\bigcup_{\psi\in \U_{\phi}} \F(\psi)\subset\Real^n$ is bounded.
\end{defn}

In the above definitions, $\F$ is said to be outer semicontinuous (respectively, locally bounded) \emph{relative} to a set $\M'\subset\MD$, if the mapping $\F\vert_{\M'}$ (defined by $\F\vert_{\M'}(\phi)=\F(\phi)$ if $\phi\in\M'$ and $\F\vert_{\M'}(\phi)=\emptyset$ elsewhere) is outer semicontinuous (respectively, locally bounded) at each $\phi\in \M'$. Finally, the mapping $\F$ is said to be outer semicontinuous (respectively, locally bounded) if it is so relative to its domain.

The following is a list of basic conditions on the data of the hybrid system $\HMD=(\C,\F,\D,\G)$.

\begin{assum}\label{as1}
For every $\bl\ge0$, the following hold:
\begin{enumerate}
\item[(A1)] $\C\cap\Mbl$ and $\D\cap\Mbl$ are closed subsets of $\MD$;
\item[(A2)] $\F$ is outer semicontinuous relative to the set
$\C\cap\Mbl$, locally bounded relative to the set
$\C\cap\M_b^{\Delta}$, and $\F(\phi)$ is nonempty and convex for each $\phi\in\C\cap\Mbl$;
\item[(A3)] $\G$ is outer semicontinuous relative to the set
$\D\cap\Mbl$, locally bounded relative to the set $\D\cap\M_b^{\Delta}$, and $\G(\phi)$ is nonempty for each $\phi\in\D\cap\Mbl$.
\end{enumerate}
\end{assum}

{\color{black} The list of basic conditions (A1)--(A3) in the above assumption are similar to those required for hybrid systems without memory \cite{goebel2012hybrid}. Not only do they provide a set of sufficient conditions for the existence of generalized solutions (Theorem \ref{thm:ext}), they also guarantee that hybrid systems are well-posed (Theorem \ref{thm:well-posedness}). }

The following definition generalizes tangent cones from a set in $\Real^n$ to that of a set in $\MD$ in order to formulate viability conditions in $\MD$. The definition is based on the definition of tangent cones in functional spaces for developing existence theory for functional differential inclusions (see, e.g., Chapter 12 of \cite{aubin1991viability}).

\begin{defn}\em 
For any $\phi\in\mathcal{K}\subset\MD$, we define $\T_{\K}(\phi)\subset\Real^n$ by $v\in \T_{\K}(\phi)$ if and only if, for any $\eps>0$, there exist $h\in (0,\eps]$ and { a linear function} $x_h\in C([0,h],\Real^n)$ such that\\
(1) $x_h(0)=\phi(0,0)$ and
$$
{ \frac{x_h(s)-x_h(0)}{s}\in v+\eps\mathbb{B},\quad \forall s\in(0,h];}
$$
(2) the hybrid memory arc $\psi_{x_h}$ defined by:
\begin{equation}
\psi_{x_h}(s,k)=\left\{\begin{array}{ll}
x_h(h+s),\quad \forall s\in [-h,0],\; k=0,\;s\ge-\Delta,\\
\phi(h+s,k),\quad\forall (h+s,k)\in\dom\phi,\;s+k\ge-\Delta_{\inf},
\end{array}\right.\label{eq:ext}
\end{equation}
lies in $\K$, where $\Delta_{\inf}:=\inf\Big\{\delta\ge\Delta:\,\exists(h+s,k)\in\dom \phi\;\text{s.t.}\;   s+k=-\delta\Big\}.
$
\end{defn}

{\color{black} Now we are ready to state and prove one of the main results of this paper. }

\section{Basic existence}\label{sec:exist}

{\color{black} The following theorem gives an existence result for generalized solutions of hybrid systems with memory. The main proof consists of  constructing a sequence of approximate solutions using a viability condition and proving that the limit of this sequence is a true solution to system. }

\begin{thm}\label{thm:ext}
Let $\HMD=(\C,\F,\D,\G)$ satisfy the conditions (A1)--(A3) in Assumption \ref{as1}. If, for every $\psi\in \C\backslash \D$,
\begin{equation}\label{eq:VC}
\F(\psi)\cap \T_\C(\psi)\neq\emptyset,
\end{equation}
then there exists a nontrivial solution to $\HMD$ from every initial condition $\phi\in{\C\cup \D}$ such that $\phi\in\MD_{b_0,\lambda_0}$ for some $b_0,\lambda_0\in\Real_{\ge 0}$. Moreover, every such maximal solution $x$ satisfies exactly one of the following conditions:
\begin{enumerate}[(a)]
\item $x$ is complete;
\item $\dom_{\ge 0}(x)$ is bounded, the interval $I^J$ has nonempty interior, and $\limsup_{t\ra T^-}\abs{x(t,J)}=\infty$, where 
$$T=\sup_{t}\dom x:=\sup\set{t:\,(t,j)\in\dom x}$$ and 
$$J=\sup_j\dom x:=\sup\set{j:\,(t,j)\in \dom x};$$
\item $\mathcal{A}_{[T,J]}^{\Delta}x\not\in \C\cup \D$, where $(T,J)$ is as defined in (b).
\end{enumerate}
Furthermore, if $\G^+(\D)\subset \C\cup \D$, then (c) above does not occur.
\end{thm}

\begin{proof}
\emph{\bfseries (Local existence)} If $\phi\in \D$, then the hybrid arc $x$ with $\AD_{[0,0]}x=\phi$ and $x(0,1)=g$ with any $g\in \G(\phi)$ provides a desired solution. Otherwise, $\phi\in \C\backslash \D$ and the viability condition (\ref{eq:VC}) is satisfied at $\phi$. Given any $a>0$, define
$$
\ms:=\set{\psi\in\C\cap \M_{b,\lambda}^{\Delta}:\;\abs{\psi(0,0)-\phi(0,0)}\le a},
$$
where $b:=b_0+a$ and $\lambda>\max(1,\lambda_0)$ is such that $\F(\psi)\subset(\lambda-1)\mathbb{B}$ for all $\psi\in\ms$.
This existence of such a $\lambda$ follows from the face that $\F$ is locally bounded and $\ms$ is a compact set in $(\MD,\d)$.

Fix any $\eps\in (0,1)$. For each $\psi\in \ms$, by the viability condition (\ref{eq:VC}), there exists $v_{\psi}\in \F(\psi)$, $h_{\psi}\in (0,\eps]$, and a hybrid arc $y_\psi$ defined on $\dom\psi\cup [0,h]\times\set{0}$ such that $\AD_{[0,0]}y_\psi=\psi$, $\AD_{[h_{\psi},0]}y_\psi\in\C$, and
$$
\frac{y_\psi(s,0)-y_\psi(0,0)}{s}\in v_{\psi}+\eps\mathbb{B},\quad\forall s\in (0,h_{\psi}].
$$
{ It follows from $v_{\psi}+\eps\mathbb{B}\subset (\lambda-1+\eps)\B\subset{\lambda}\B$ that $y_\psi$ is $\lambda$-Lipschitz.}\\
\emph{\bfseries Claim A:} Define
$$
E(\psi):=\Big\{\xi\in\ms:\,\xi\text{ and }\psi\text{ are }(1/\eps,h_{\psi}\eps)\text{-close}\Big\},
$$
where the definition for $(1/\eps,h_{\psi}\eps)\text{-closeness}$ is given in Section \ref{app:distance}. There exists a finite index set $I$ and $\set{\psi_i\in\ms}_{i\in I}$ such that
\begin{equation}\label{eq:epsnet}
\ms\subset\bigcup_{i\in I}E(\psi_i).
\end{equation}
\emph{\bfseries Proof of Claim A:} By Proposition
\ref{dprop}(c) in Section \ref{app:distance}, if $\phi\in\M_S$ is such that $\d(\phi,\psi)<\delta$, with $\delta e^{\rho}<h_\psi\eps<1$, $\rho=\sqrt{b^2+\tau^2}$, and $\tau>1/\eps$, then we know $\phi$ and $\psi$ are $(1/\eps,h_{\psi}\eps)$-close and hence $\phi\in E(\psi)$. This shows $\set{\phi\in\M_S:\,\d(\phi,\psi)<\delta}\subset E(\psi)$. By compactness of $\M_S$ guaranteed by Theorem \ref{thm:compact}, the conclusion of the claim holds. $\blacksquare$

\emph{\bfseries Claim B:} Given $a$ and $\lambda$, for each $\eps\in (0,1)$, there exists positive numbers $\set{h_k}_{k=1}^p$, real vectors $\set{v_k}_{k=1}^p$,  and hybrid arcs $\set{y_k}_{k=1}^p$ such that $\sum_{k=1}^{p-1}h_k\le\frac{a}{\lambda+(1+\lambda)\eps}<\sum_{k=1}^{p}h_k$ and
\begin{equation}
\left\{
\begin{array}{@{}ll}
&\AD_{[0,0]}y_k\in \ms,\quad \AD_{[h_k,0]}y_k\in \C,\quad v_k\in\F(\AD_{[0,0]}y_k),\\[2\jot]
&{\displaystyle\frac{y_k(h_k,0)-y_k(0,0)}{h_k}\in v_{k}+\eps\mathbb{B}},\\
&\text{$\AD_{[h_{k-1},0]}y_{k-1}$ and $\AD_{[0,0]}y_k$ are $(1/\eps,h_k\eps$)-close,}
\end{array}
\right.\label{eq:const}
\end{equation}
holds for all $k=1,\cdots,p$, where $y_0=\phi$, $h_0=0$, and the domain of each $y_k$, $k=1,\cdots,p$, contains
$[0,h_k]\times\set{0}$. \\
\emph{\bfseries Proof of Claim B:} Clearly, $y_0=\phi\in \ms$. By (\ref{eq:epsnet}) and the argument that precedes it, there exists $i_1\in I$ such that $y_0\in E(\phi_{i_1})$. Denote $h_1=h_{\phi_{i_1}}$, $y_1=y_{\phi_{i_1}}$, and $v_1=v_{\phi_{i_1}}$. Note that $h_0=0$, $\AD_{[0,0]}y_1=\phi_{i_1}$, $\AD_{[h_0,0]}y_0=y_0=\phi$, and hence $\AD_{[h_{0},0]}y_{0}=\phi$ and $\AD_{[0,0]}y_1=\phi_{i_1}$ are $(1/\eps,h_1\eps)$-close. The rest of (\ref{eq:const}) for $k=1$ follows from the argument that precedes Claim A. Consider
\begin{align*}
\abs{y_1(s,0)-\phi(0,0)}&\le \abs{y_1(s,0)-y_1(0,0)}+\abs{y_1(0,0)-\phi(0,0)}\\
&\le \lambda s +(1+\lambda)h_1\eps \le \lambda+(1+\lambda)\eps]h_1,\quad\forall s\in [0,h_1],
\end{align*}
where $$\abs{y_1(s,0)-y_1(0,0)}\le \lambda s$$ follows from the fact that $y_1$ is $\lambda$-Lipschitz and 
$$\abs{y_1(0,0)-\phi(0,0)}<(1+\lambda)h_1\eps$$ follows from the fact that $\phi$ and $\AD_{[0,0]}y_1$ are $(1/\eps,h_1\eps)$-close and $\lambda$-Lipschitz. If $[\lambda+(1+\lambda)e^{b}\eps]h_1>a$, we stop. Otherwise, we have $y_1(s,0)\in \phi(0,0)+a\mathbb{B}$ for all $s\in [0,h_1]$. This in turn implies that $\norm{\AD_{[h_1,0]}y_1}\le b$. It further follows that $\AD_{[h_1,0]}y_1\in\ms$. Then there exists $\phi_{i_2}$ such that $\AD_{[h_1,0]}y_1\in E(\phi_{i_2})$. Denote $h_2=h_{\phi_{i_2}}$, $y_2=y_{\phi_{i_2}}$, and $v_2=v_{\phi_{i_2}}$. It can be verified that $h_2$, $y_2$, and $v_2$ satisfy (\ref{eq:const}). Moreover,
\begin{align*}
\abs{y_2(h_2,0)-\phi(0,0)}&\le \abs{y_2(h_2,0)-y_1(h_1,0)}+\abs{y_1(h_1,0)-\phi(0,0)},
\end{align*}
where $\abs{y_1(h_1,0)-\phi(0,0)}\le [\lambda+(1+\lambda)\eps]h_1$ is shown earlier and
\begin{align*}
\abs{y_2(h_2,0)-y_1(h_1,0)}&\le \abs{y_2(h_2,0)-y_2(0,0)}+\abs{y_2(0,0)-y_1(h_1,0)}\\
&\le \lambda h_2 + (1+\lambda)h_2\eps,
\end{align*}
where $\abs{y_2(h_2,0)-y_2(0,0)}\le \lambda h_2$ follows from the fact that $y_2$ is $\lambda$-Lipschitz and $\abs{y_2(0,0)-y_1(h_1,0)}<(1+\lambda)h_2\eps$ follows from the fact that $\AD_{[h_1,0]}y_1$ and $\AD_{[0,0]}y_2$ are $(1/\eps,h_2\eps)$-close and $\lambda$-Lipschitz.. The above combined gives
\begin{align*}
\abs{y_2(h_2,0)-\phi(0,0)}&\le[\lambda+(1+\lambda)\eps](h_1+h_2).
\end{align*}
We stop if $[\lambda+(1+\lambda)\eps](h_1+h_2)>a$. Otherwise, $\AD_{[h_2,0]}y_2\in\ms$, and we can continue the above procedure until we find a finite number of $h_{k}\in \set{h_{\phi_i}:\,i\in I}$, $k=1,\cdots,p$, such that $\sum_{k=1}^{p-1}h_k\le\frac{a}{\lambda+(1+\lambda)\eps}<\sum_{k=1}^{p}h_k$. The claim is proved. $\blacksquare$\\
\emph{\bfseries Construction of Approximate Solutions:}
Define a hybrid arc $y_\eps$ by $\AD_{[0,0]}y_\eps=y_0=\phi$ and
\begin{align*}
y_\eps(s,0)=y_{i+1}(s-\sum_{k=0}^{i}h_k,0)+\sum_{k=0}^i[y_k(h_k,0)-y_{k+1}(0,0)],
\end{align*}
if
\begin{align*}
s\in [\sum_{k=0}^{i}h_k,\sum_{k=0}^{i+1}h_k],\quad i\in\set{0,\cdots,p-1}.
\end{align*}
We further define a hybrid arc $x_\eps$ by $\AD_{[0,0]}x_\eps=\AD_{[0,0]}y_\eps=y_0=\phi$ and
\begin{align*}
x_\eps(s,0)=\frac{s-\sum_{k=0}^{i}h_k}{h_{i+1}}
\Big[y_\eps(\sum_{k=0}^{i+1}h_k,0)-y_\eps(\sum_{k=0}^{i}h_k,0)\Big]
+y_\eps(\sum_{k=0}^{i}h_k,0),
\end{align*}
if
\begin{align*}
s\in [\sum_{k=0}^{i}h_k,\sum_{k=0}^{i+1}h_k],\quad i\in\set{0,\cdots,p-1}.
\end{align*}
Note that $\dom x_\eps=\dom y_\eps=\dom\phi \cup [0,\sum_{k=0}^{p}h_k]\times\set{0}$. We can check that both $x_\eps$ and $y_\eps$ are $\lambda$-Lipschitz and $\norm{y_\eps-x_\eps}<\lambda\eps$. Moreover,
\begin{align}
\dot{x}_\eps(s,0)&=\frac{x_\eps(\sum_{k=0}^{i+1}h_k,0)-x_\eps(\sum_{k=0}^{i}h_k,0)}{h_{i+1}}\notag\\
&=\frac{y_\eps(\sum_{k=0}^{i+1}h_k,0)-y_\eps(\sum_{k=0}^{i}h_k,0)}{h_{i+1}}\notag\\
&=\frac{y_{i+1}(h_{i+1},0)-y_{i+1}(0,0)}{h_{i+1}}\in v_{i+1}+\eps\B,\label{eq:dev}
\end{align}
with $v_{i+1}\in\F(\AD_{[0,0]}y_{i+1})$, for all
\begin{align*}
s\in [\sum_{k=0}^{i}h_k,\sum_{k=0}^{i+1}h_k],\quad i\in\set{0,\cdots,p-1}.
\end{align*}

\textbf{\emph{Claim C:}} Suppose $\eps>0$ is sufficiently small such that $1/\eps>(1+\lambda)\sum_{k=0}^{p}h_k\eps$. The following holds for all $i\in\set{0,\cdots,p-1}$:
\begin{align}\label{eq:xeps}
\text
{$\AD_{[\sum_{k=0}^{i+1}h_k,0]}x_\eps$ and $\AD_{[h_{i+1},0]}y_{i+1}$
are $(1/\eps-(1+\lambda)\sum_{k=0}^{i+1}h_k\eps,(1+\lambda)\sum_{k=0}^{i+1}h_k+\lambda\big]\eps)$-close}.
\end{align}

\noindent\emph{\bfseries Proof of Claim C:} We note that (\ref{eq:xeps}) will follow from
\begin{align}\label{eq:xeps2}
\text
{$\AD_{[\sum_{k=0}^{i+1}h_k,0]}y_\eps$ and $\AD_{[h_{i+1},0]}y_{i+1}$
are $(1/\eps-(1+\lambda)\sum_{k=0}^{i+1}h_k\eps,(1+\lambda)\sum_{k=0}^{i+1}h_k\eps)$-close}.
\end{align}
since $\norm{y_\eps-x_\eps}<\lambda\eps$. We prove (\ref{eq:xeps2}) by induction on $i$. Starting with $i=0$, note that
\begin{equation}
\label{eq:est11}
\abs{y_\eps(s,0)-y_1(s,0)}\le \abs{y_1(0,0)-\phi(0,0)}\le (1+\lambda)h_1\eps,\quad\forall s\in [0,h_1].
\end{equation}
We know that $\AD_{[0,0]}y_\eps=\phi$ and $\AD_{[0,0]}y_{1}=\phi_{i_1}$ are already $(1/\eps,h_1\eps)$-close.
It follows that $\AD_{[h_1,0]}y_\eps$ and $\AD_{[h_1,0]}y_{1}$ are $(1/\eps,(1+\lambda)h_1\eps)$-close. Thus, we have proved (\ref{eq:xeps2}) for $i=0$. Now suppose that (\ref{eq:xeps2}) holds for some $i$. Since $\AD_{[h_{i+1},0]}y_{i+1}$ and $\AD_{[0,0]}y_{i+2}$ are $(1/\eps,h_{i+2}\eps)$-close. It follows from Lemma \ref{lem:tri} that $\AD_{[0,0]}y_{i+2}$ and $\AD_{[\sum_{k=0}^{i+1}h_k,0]}y_\eps$ are $(\tau,(1+\lambda)\sum_{k=0}^{i+1}h_k\eps+h_{i+2}\eps))$-close with $\tau=\min(1/\eps-(1+\lambda)\sum_{k=0}^{i+1}h_k\eps-h_{i+2}\eps,1/\eps-(1+\lambda)\sum_{k=0}^{i+1}h_k\eps)$. This implies $\AD_{[0,0]}y_{i+2}$ and $\AD_{[\sum_{k=0}^{i+1}h_k,0]}y_\eps$ are $(1/\eps-(1+\lambda)\sum_{k=0}^{i+2}h_k\eps,(1+\lambda)\sum_{k=0}^{i+2}h_k\eps)$-close. In addition, note that for all $s\in[\sum_{k=0}^{i+1}h_k,\sum_{k=0}^{i+2}h_k]$, we have
$$
\abs{y_\eps(s,0)-y_2(s-\sum_{k=0}^{i+1}h_k,0)}\le \sum_{k=0}^{i+1} \abs{y_k(h_k,0)-y_{k+1}(0,0)}\le (1+\lambda)\sum_{k=0}^{i+2}h_k\eps.
$$
It follows that (\ref{eq:xeps2}) holds for $i+1$. The claim is proved. $\blacksquare$

\noindent\emph{\bfseries Convergence to a True Solution:}
Given any $T_0<\frac{a}{\lambda}$, choose a strictly decreasing sequence $\set{\eps_n}_{n=1}^{\infty}$ such that $T_0<\frac{a}{\lambda+(1+\lambda)\eps_1}$ and $\eps_n\ra 0$ as $n\ra\infty$. The sequence of hybrid arcs $X_n:=x_{\eps_n},$
$n=1,2,\cdots$, are defined on $\dom\phi\cup([0,T_0]\times\set{0})$ and satisfy $\AD_{[0,0]}X_{n}=\phi$ for all $n$. Moreover, each $X_{n}(\cdot,0)$ is $\lambda$-Lipschitz on $[0,T_0]$. By Ascoli's theorem, there exists a subsequence of $X_{n}(\cdot,0)$ (still denoted by $X_{n}$) that converges uniformly to a function $Y$ on $[0,T_0]$. We can define a hybrid arc $X$ with domain $\dom\phi\cup[0,T_0]\times\set{0}$ and $\AD_{[0,0]}X=\phi$. Moreover, $X(\cdot,0)$ is also $\lambda$-Lipschitz on $[0,T_0]$ and hence $\dot{X}(\cdot,0)$ exists almost everywhere on $[0,T_0]$ and $\dot{X}(\cdot,0)\in L^{\infty}([0,T_0],\Real^n)$.

For each $\eps_n$, let $\set{y_i}_{i=0}^{p_n}$ still denote the associated sequence constructed on the series of intervals $\big\{[\sum_{k=0}^{i}h_k,\sum_{k=0}^{i+1}h_k]\big\}_{i=0}^{p_n-1}$. Fix any $\rho>0$.

\noindent\emph{\bfseries Claim D:} There exists $N_1>0$ such that
\begin{equation}\label{estrho1}
\d\Big(\AD_{[\sum_{k=0}^{i}h_k,0]}X_{n},\AD_{[0,0]}y_{i+1}\Big)<\rho/3
\end{equation}
for all $n>N$ and all $i\in \set{0,\cdots,p_n}$.

\noindent\emph{\bfseries Proof of Claim D:} It follows from Claim C that
$$
\dt_{\tau}\Big(\AD_{[\sum_{k=0}^{i}h_k,0]}X_{n},\AD_{[h_{i},0]}y_{i}\Big)
\le \big[(1+\lambda)\sum_{k=0}^{i}h_k+\lambda\big]\eps_n<\rho/3,
$$
with $\tau=1/\eps_n-(1+\lambda)\sum_{k=0}^{i}h_k\eps_n$. Moreover, by how we construct $\set{y_i}_{i=0}^{p}$, we have
$$\dt_{1/\eps_n}(\AD_{[0,0]}y_{i+1},\AD_{[h_i,0]}y_{i})\le h_{i+1}\eps_n.$$
It follows from Lemma \ref{lem:tri} that
\begin{align}
\dt_{1/\eps_n-\bar{h}\eps_n}\Big(\AD_{[\sum_{k=0}^{i}h_k,0]}X_{n},\AD_{[0,0]}y_{i+1}\Big)
&\le\bar{h}\eps_n,\label{eq:est1}
\end{align}
where $\bar{h}$ is a constant that upper bounds $(1+\lambda)\sum_{k=0}^{i+1}h_k+\lambda$ for all $i$. The conclusion of the claim follows from  Proposition \ref{dprop}(c), as we can choose $N_1$ sufficiently large (hence $\tau$ and $\bar{\rho}$ there sufficiently large) such that the right-hand side of (\ref{taueps}) is less than $\rho/3$. $\blacksquare$

As $X_{n}(\cdot,0)$ converges uniformly to a function $X(\cdot,0)$ on $[0,T_0]$ and $\AD_{[0,0]}X_n=\AD_{[0,0]}X_n=\phi$, it follows from Proposition \ref{dprop}(a) that there exists $N_2>0$ such that
\begin{equation}\label{estrho2}
\d\Big(\AD_{[\sum_{k=0}^{i}h_k,0]}X_{n},\AD_{[\sum_{k=0}^{i}h_k,0]}X\Big)<\frac{\rho}{3}
\end{equation}
holds for all $n>N_2$ and all $i\in \set{0,\cdots,p_n}$ such that $\sum_{k=0}^{i}h_k\le T_0$.

Lemma \ref{lem:cont} shows that $\AD_{[t,0]}X$ is uniformly continuous in $t$ on $[0,T_0]$. Let $\delta>0$ be such that $t',t''\in [0,T_0]$ and $\abs{t'-t''}<\delta$ imply
\begin{equation}\label{estrho3}
\d(\AD_{[t',0]}X,\AD_{[t'',0]}X)<\frac{\rho}{3}.
\end{equation}

Fix $t\in (0,T_0)$. Choose $N_3>0$ sufficiently large such that $t\in[\sum_{k=0}^{i}h_k,\sum_{k=0}^{i+1}h_k]\subset (0,T_0)$ and $h_{i+1}\le\eps_n<\delta$ hold for all $n>N_3$.

Let $N=\max_{1\le i\le 3} N_i$. Combing (\ref{estrho1}), (\ref{estrho2}), and (\ref{estrho3}) gives
\begin{align}\label{estrho4}
\d\Big(\AD_{[0,0]}y_{i+1},\AD_{[t,0]}X\Big)<\rho,\quad \forall n>N.
\end{align}

To show that $X$ is a solution to $\H_\M$, we have to show that
\begin{equation}\label{eq:final}
\dot{X}(t,0)\in\F(\AD_{[t,0]}X),\quad \text{for almost all}\quad t\in (0,T_0).
\end{equation}

We first prove the following.

\noindent\emph{\bfseries Claim E:} Given any $\eta>0$, for each $t\in [0,T_0]$, there exists $N'$ sufficiently large such that
\begin{align}\label{esteta}
\dot{X}_{n}(t,0)\in \F(\AD_{[t,0]}X) +\eta\B,
\end{align}
holds for all $n\ge N'$.

\noindent\emph{\bfseries Proof of Claim E:} Choose $\rho>0$ such that $\d(\psi,\AD_{[t,0]X})\le\rho$ implies $\F(\psi)\subset\F(\AD_{[t,0]}X)+\frac{\eta}{2}\B$. Let $N$ be chosen such that (\ref{estrho4}) holds and hence
\begin{align}\label{esteta0}
\F(\AD_{[0,0]}y_{i+1})\subset \F(\AD_{[t,0]}X)+\frac{\eta}{2}\B,\quad \forall n>N.
\end{align}
Note that $\rho$ may depend on $t$ and hence $N$ chosen above may also depend on $t$. Furthermore, choose $N'>N$ sufficiently large such that $n>N'$ implies $\eps_n<\frac{\eta}{2}$. The conclusion of the claim now follows from (\ref{eq:dev}) and (\ref{esteta}). $\blacksquare$

Note that for all $\underline{t},\,\bar{t}\in[0,T_0]$ we have
$
\int_{\underline{t}}^{\bar{t}}\dot{X}_n(s,0)ds=X_n(\bar{t},0)-X_n(\underline{t},0),
$
which converges to
$
X(\bar{t},0)-X(\underline{t},0)=\int_{\underline{t}}^{\bar{t}}\dot{X}(s,0)ds
$
as $n\ra\infty$. Since the derivatives $\dot{X}_n(t,0)$ are equibounded on $[0,T_0]$, we conclude from $L^{\infty}([0,T_0],\Real^n)\subset L^{1}([0,T_0],\Real^n)$ that the sequence $\dot{X}_n(t,0)$ converters weakly to $\dot{X}(t,0)$ in $L^{1}([0,T_0],\Real^n)$. Using Mazur's convexity theorem \cite{yosida1995functional}, we can construct a sequence
\begin{equation}\label{mazur}
Z^l(t)=\sum_{n=l}^{N(l)}\alpha^n_{l}\dot{X}_n(t,0),\quad t\in [0,T_0],\quad \alpha^n_{l}\ge 0,\quad \sum_{n=l}^{N(l)}\alpha^n_{l}=1
\end{equation}
such that $Z^l$ converges strongly to $\dot{X}(\cdot,0)$ in $L^{1}([0,T_0],\Real^n)$ as $l\ra \infty$. Then we can extract a subsequence of $Z^l$ (still denoted by $Z^l$) that converges to $\dot{X}(\cdot,0)$ pointwise for almost all $t\in(0,T_0)$. From (\ref{esteta}) and (\ref{mazur}) and that $\F$ is convex valued, we conclude that for large enough $l$ that $Z^l(t)\in \F(\AD_{[t,0]}X)+\eta\B$. Since $\F(\AD_{[t,0]}X)+\eta\B$ is closed, taking the limit as $l\ra\infty$ implies that $\dot{X}(\cdot,0)\in \F(\AD_{[t,0]}X)+\eta\B$, which holds for almost all $t\in(0,T_0)$. Since $\eta>0$ is arbitrary, we actually proved (\ref{eq:final}).

Finally, we prove that $\AD_{[t,0]}X\in \C$ for all $[0,T_0]$. By (\ref{eq:const}) of Claim B, we have $\AD_{[0,0]y_i}\in \M_S\subset\C$ for all $i\in\set{0,\cdots,p_n}$. For each $t\in(0,T_0)$, it follows from (\ref{estrho4}) that there exists a sequence $\phi_n\in \M_S$ such that $\d(\phi_n,\AD_{[t,0]}X)\ra 0$ as $n\ra \in\infty$. Since $\M_S$ is a closed set in the graphical convergence topology, we know $\AD_{[t,0]}X\in\M_S$ for all $t\in(0,T_0)$. By continuity of $\AD_{[t,0]}X$ in $t$ on $[0,T_0]$ (Lemma \ref{lem:cont}), $\AD_{[t,0]}X\in\M_S$ for all $t\in(0,T_0)$.

{\color{black} \emph{\bfseries Verifying (a)--(c):} The argument is similar to that in the proof of Proposition 2.10 in \cite{goebel2012hybrid}. Suppose that case (a) is not true and $x$ is a maximal solution that is not complete, i.e., $\dom_{\ge 0}(x)$ is bounded. If $(T,J)\in\dom(x)$ and $\AD_{[T,J]}x\in \C\cup \D$, using $\AD_{[T,J]}x$ as initial data, the local existence result we have just proved would allow us to extend the solution either by a jump or a flow. This would contradict the maximality of $x$. Therefore, either $\AD_{[T,J]}x\not\in \C\cup \D$, which verifies (c), or $(T,J)\not\in\dom(x)$. In the latter case, $I^J$ would have nonempty interior. Suppose $I^J=[\alpha,T)$. If $\limsup_{t\ra T^-}\abs{x(t,J)}=\infty$ does not hold, then again the solution can be extended beyond $[\alpha,T) $ by a flow.}
\end{proof}

\section{Well-posedness}
\label{sec:well-posed}

In order to discuss robustness of stability, we define perturbations of a hybrid system with memory as follows. The definition presented here follows closely the notion of outer perturbations of a hybrid system without memory \cite{goebel2012hybrid}, but formulated in a more restricted sense by making the following  assumption on $\HMD$: there exists a nondecreasing function $\lambda:\,\Real_{\ge 0}\ra\Real_{\ge 0}$ such that
{ $$
\C\cup\D\subset\wMl:=\bigcup_{b\in\Real_{\ge 0}}\MD_{b,\lambda(b)}.
$$}
{\color{black} We define perturbations of hybrid data within this set $\wMl$ as follows.}

\begin{defn}

Given a hybrid system with memory $\HMD=(\C,\F,\D,\G)$ and a functional $\rho:\,\MD\ra \Real_{\ge 0}$, the $\rho$-perturbation of $\HMD$ with a function $\lambda$ given in the above assumption, denoted by $(\HMD)_{\rho}$, is the hybrid system with data:
\begin{itemize}
\item $\C_\rho=\set{\phi\in\wMl:\,\wB(\phi,\rho(\phi))\cap \C\neq\emptyset}$;
\item $\F_\rho(\phi)=\overline{\text{con}}\F(\wB(\phi,\rho(\phi))\cap \C)+\rho(\phi)\mathbb{B}$;
\item $\D_\rho=\set{\phi\in\wMl:\,\wB(\phi,\rho(\phi))\cap \D\neq\emptyset}$;
\item $\G_\rho(\phi)= \set{y\in\Real^n:\,y\in v + \rho(\A_{[0,1]}^{\Delta}\phi^+_v)\mathbb{B},\,v\in \G(\wB(\phi,\rho(\phi))\cap \D)}$, where $\phi^+_v$ (the hybrid arc after one jump from $\phi$ with value $v$) is defined by $\gph \phi^+_v := \gph\phi\cup\set{(0,1,v)}$ and the operator $\A$ is as defined in (\ref{eq:opA}). 
\end{itemize}
where $\overline{\text{con}}(E)$ is the closed convex hull of a set $E\subset\Real^n$ and
$$\wB(\phi,\rho(\phi)):=\set{\psi\in\wMl:\,\psi \;\text{and}\; \phi\;\text{are $\rho(\phi)$-close}},$$
where $\rho(\phi)$-closeness is defined in Section \ref{app:distance}, and $\mathbb{B}$ is the closed unit ball in $\Real^n$.
\end{defn}

\begin{prop}
Let $\rho$ be a continuous functional $\rho:\,\MD\ra \Real_{\ge 0}$ and $\HMD=(\C,\F,\D,\G)$ be a hybrid system satisfying Assumption \ref{as1}. Then $(\HMD)_\rho=(\C_\rho,\F_\rho,\D_{\rho},\G_{\rho})$ satisfy Assumption \ref{as1}.
\end{prop}

\begin{proof}
Let $\set{\phi_i}_{i=1}^{\infty}$ be a graphically convergent sequence in $\C_{\rho}\cap\MD_{b_0,\lambda_0}$ for some $b_0\ge 0$ and $\lambda_0\ge 0$. It follows from Proposition \ref{thm:compact} that $\phi_i\rap\phi$ for some $\phi\in \MD_{b_0,\lambda}$, where $\lambda:=\min(\lambda(b),\lambda_0)$. In addition, by Proposition \ref{dprop} in Section \ref{app:distance}, there exists a sequence $\eps_n\ra 0$ and a sequence $N_n\ra\infty$ such that $\phi_i$ and $\phi$ are $\eps_n$-close for all $i>N_n$. By definition of $\C_\rho$, for each $\phi_i$, there exists $\psi_i\in \C$ such that $\phi_i$ and $\psi_i$ are $\rho(\phi_i)$-close for all $i\ge 1$. 
Since $\rho(\phi_i)\ra\rho(\phi)$, it follows that there exists $b'\ge 0$ such that $\psi_i\in \C\cap\MD_{b'}$ for all $i\ge 1$. It follows from the closedness of $\C\cap\MD_{b'}$ that $\psi_i\ra\psi$ for some  $\psi\in \C\cap\MD_{b'}$. Again, by Proposition \ref{dprop} in Section \ref{app:distance}, there exists a sequence $\eps'_n\ra 0$ and a sequence $N'_n$ such that $\psi_i$ and $\psi$ are $\eps'_n$-close for all $i>N'_n$. It follows that $\psi$ and $\phi$ are $(\eps_n+\eps'_n+\rho(\phi_i))$-close for all $i>\max(N_n,N'_n)$. Letting $n\ra\infty$ implies that $\psi$ and $\phi$ are $\rho(\phi)$-close. Since $\psi\in \C\cap \wMDl$, we have $\phi\in\C_{\rho}$ by definition. 

To prove local boundedness of $\F_{\rho}$ relative to $\C_{\rho}$, fixed any $\phi\in \C_{\rho}$, $b\ge 0$, and a neighbourhood $\U_{\phi}$ of $\phi$. We show that $\F_{\rho}(\U_{\phi}\cap \MD_{b}\cap\C_{\rho})$ is bounded. Note that the set $\U:=\U_{\phi}\cap \MD_{b}\cap \C_{\rho}\subset \MD_{b,\lambda(b)}$ is a compact set in $(\MD,\d)$. It follows from the continuity of $\rho$ that $\rho$ is bounded on $\U$ and hence the set $\wB(\psi,\rho(\psi))\cap \C$ for $\psi\in \U$ belongs to a set $\MD_{b'}$ for some $b'>b\ge 0$. Therefore, $\wB(\psi,\rho(\psi))\cap \C\subset \MD_{b',\lambda(b')}$ for all $\psi\in\U$. Since $\MD_{b',\lambda(b')}$ is a compact set, it follows that from the local boundedness  of $\F$ relative to $\C$ that $\F$ is bounded on $\bigcup_{\psi\in \U}\wB(\psi,\rho(\psi))\cap \C$.

To prove outer semicontinuity of $\F_{\rho}$ relative to $\C_{\rho}\cap\MD_{b}$, letting $\set{\phi_i}_{i=1}^{\infty}$ be a graphically convergent sequence in $\C_{\rho}\cap\MD_{b}$ and a sequence $\set{y_i}_{i=1}^{\infty}$ such that $y_i\in \F_{\rho}(\phi_i)$ and $y_i\ra y$. The goal is to show that $y\in \F_{\rho}(\phi)$. By the definition of $\F_{\rho}$, there exists a sequence $\eps_i\ra 0$ such that, for each $i$, $y_i=\sum_{l=1}^{n+1}\mu_i^{l}u_i^{l}+\eps_i+v_i$, where $\sum_{l=1}^{n+1}\mu_i^{l}=1$, $\mu_i^{l}\in[0,1]$, $u_i^{l}\in \F(\psi_i^l)$, $\psi_i^l\in \wB(\phi_i,\rho(\phi_i))\cap\C$, and $v_i\in\rho(\phi_i)\mathbb{B}$. By local boundedness of $\F_{\rho}$ and continuity of $\rho$, the sequences $\mu_i^{l}$, $u_i^{l}$, and $v_i$ are bounded. Moreover, $\psi_i^l\in \C\cap \wB(\phi_i,\rho(\phi_i))\subset \C\cap \MD_{b}$ for some $b>0$. There exists a subsequence of $\set{\phi_i}_{i=1}^{\infty}$ such that the corresponding subsequences of $\mu_i^{l}$, $u_i^{l}$, $\psi_i^l$, and $v_i$ converges to $\mu^l$ with $\sum_{l=1}^{n+1}\mu^{l}=1$, $u\in \F(\psi)$, $\psi\in \C\cap \MD_{b}\cap \wB(\phi,\rho(\phi))$, and $v\in\rho(\phi)\mathbb{B}$, respectively. It follows that, without relabelling the subsequences, $y=\sum_{l=1}^{n+1}\mu^{l}u^{l}+v\in \text{con}\F(\psi)+\rho(\phi)\mathbb{B}\subset \F_{\rho}(\phi)$. The proof for the properties of $\D$ and $\G_{\rho}$ are similar. 
\end{proof}

\begin{defn}[Well-posedness of hybrid systems with memory]\em
A hybrid system $\HMD$ is said to be \emph{well-posed} if the following properties hold: for any given  continuous function $\rho:\,\MD\ra\Real_{\ge 0}$, a decreasing sequence $\set{\delta_i}_{i=1}^\infty$ in $(0,1)$ with $\delta_i\ra 0$ as $i\ra\infty$, and for every graphically convergent sequence $\set{x_i}_{i=1}^\infty$ of solutions to $(\HMD)_{\delta_i\rho}$ with { $\AD_{[0,0]}x_i\in \Mbl$ for some $\bl\ge 0$} and $\AD_{[0,0]}x_i\rap \phi\in\MD$,
\begin{enumerate}[(a)]
\item if the sequence $\set{x_i}_{i=1}^\infty$ is locally eventually bounded (that is, for any $m>0$, there exists $N>0$ and $k>0$ such that, for all $i>N$ and all $(t,j)\in\dom x_i$ with $t+j<m$, $|x_i(t,j)|<k$), then its graphical limit $x$ is a solution to $\HMD$ with $\AD_{[0,0]}x=\phi$ and $\dom x=\lim_{i\ra\infty}\dom x_i$;
\item if the sequence $\set{x_i}_{i=1}^\infty$ is not locally eventually bounded, then there exist some $T,\,J\in(0,\infty)$ and a sequence $\set{t_i}_{i=1}^\infty$ with $(t_i,J)\in\dom{x_i}$ for sufficiently large $i$ such that $\lim_{i\ra\infty}t_i=T$, $\lim_{i\ra\infty}\abs{x_i(t_i,J)}=\infty$, and the limit $x=\limg_{i\ra\infty} x_i$ restricted to the domain $\set{(t,j)\in\dom x:\,t+j<T+J}$ is a maximal solution to $\HMD$ with $\AD_{[0,0]}x=\phi$ and $\lim_{t\ra T}\abs{x(t,J)}=\infty$.
\end{enumerate}
\end{defn}

\begin{thm}\label{thm:well-posedness}
If a hybrid system with memory $\HMD=(\C,\F,\D,\G)$ satisfies Assumption \ref{as1}, then it is well-posed.
\end{thm}

\begin{proof}
Let $\set{z_i}_{i=1}^{\infty}$ be a graphically convergent sequence of solutions to $(\HMD)_{\delta_i\rho}$ with limit $z$. If $\set{z_i}_{i=1}^{\infty}$ is locally eventually bounded, { it follows from the proof of Lemma \ref{lem:conv} that $\dom z=\lim_{i\ra \infty}\dom z_i$ is a hybrid time domain and $z$ is single-valued and locally Lipschitz on each $I_j=\set{t:\,(t,j)\in\dom z}$}. Moreover, $\set{z_i}_{i=1}^{\infty}$ converges uniformly on each compact subinterval of $\text{int}(I_j)$. We need to show that the limit $z$ is a solution to $\HMD$. To do so, we have to check that $z$ satisfies conditions (S1) and (S2) in Definition \ref{def:sol}, which are the flow and jump constraints, respectively.

\emph{\bfseries Proof of (S1):} We first prove the following. Let $I$ be any given compact subinterval of $\text{int}(I_j)$.

\noindent\emph{\bfseries Claim I:} Given any $\eta>0$, there exists $N$ sufficiently large such that
\begin{align}\label{esteta2}
\dot{z}_{n}(t,j)\in \F(\AD_{[t,j]}z) +\eta\B,
\end{align}
holds for almost all $t\in I$ and all $n\ge N$.

\noindent\emph{\bfseries Proof of Claim I:} Note that $\dot{z}_{n}(t,j)\in \F_{\delta_n\rho}(\AD_{[t,j]}z_n)=\overline{\text{con}}\F(\wB(\AD_{[t,j]}z_n,\delta_n\rho(\AD_{[t,j]}z_n))\cap \C)+\delta_n\rho(\AD_{[t,j]}z_n)\mathbb{B}$. First of all, from Proposition \ref{dprop}(c), we have
\begin{align}\label{eq:wB}
\wB(\AD_{[t,j]}z_n,\delta_n\rho(\AD_{[t,j]}z_n)\subset\B(\AD_{[t,j]}z_n,\sqrt{2}\delta_n\rho(\AD_{[t,j]}z_n).
\end{align}
Without loss of generality, we can assume that $\AD_{[t,j]}z_n\in \Mbl$ for some $\bl\ge 0$. For $\bl\ge 0$, there exists $\beta>0$ such that, for any  $\psi\in\Mbl$ satisfying $\d(\psi,\AD_{[t,j]}z)\le\beta$,  $\F(\psi)\subset\F(\AD_{[t,j]}z)+\eta/2\B$. This is possible because $\F$ is outer semicontinuous and locally bounded (which implies upper semicontinuity). Moreover, since $\set{\AD_{[t,j]}z:t\in I}$ belongs to a compact subset of $(\M,\d)$, we can choose $\beta$ independent of $t\in I$. Note that $\rho$ is bounded on the compact set $\set{\AD_{[t,j]}z:t\in I}$. Moreover, 
{ Part of 1) of Lemma \ref{lem:conv}} shows that $\d(\AD_{[t,j]}z_n,\AD_{[t,j]}z)\ra 0$ for all $t\in I$. 
Therefore, we can choose $N$ such that $\d(\AD_{[t,j]}z_n,\AD_{[t,j]}z)\le \beta/2$ and  $\sqrt{2}\delta_n\rho(\AD_{[t,j]}z_n)\le \min(\eta/2,\beta/2)$ for all $n\ge N$ and $t\in I$. It follows that
\begin{equation}\label{eq:inc}
\mathcal{B}(\AD_{[t,j]}z_n,\sqrt{2}\delta_n\rho(\AD_{[t,j]}z_n))\subset \mathcal{B}(\AD_{[t,j]}z,\beta/2+\beta/2)\subset  \mathcal{B}(\AD_{[t,j]}z,\beta).
\end{equation}
Hence,
\begin{align}
\F_{\delta_n\rho}(\AD_{[t,j]}z_n)&\subset \overline{\text{con}}\F(\mathcal{B}(\AD_{[t,j]}z_n,\delta_n\rho(\AD_{[t,j]}z_n)))+\delta_n\rho(\AD_{[t,j]}z_n)\mathcal{B}\notag\\&\subset\overline{\text{con}}\F(\mathcal{B}(\AD_{[t,j]}z,\beta))+\eta/2\B \subset \F(\AD_{[t,j]}z)+\eta\B.\label{eq:fdelta}
\end{align}
We have proved (\ref{esteta2}). $\blacksquare$

Based on (\ref{esteta2}), we can prove $\dot{z}(t,j)\in \F(\AD_{[t,j]}z)$ for almost all $t\in I$. This is similar to showing $\dot{X}(t,0)\in \F(\AD_{[t,0]}X)$ in the proof of Theorem \ref{thm:ext}.

Now we prove that $\AD_{[t,j]}z\in \C$ for all $t\in I$. Fix any $t\in I$. Since $\AD_{[t,j]}z_n\in \C_{\delta_n\rho}$, it follows that
$$\mathcal{B}(\AD_{[t,j]}z_n,\sqrt{2}\delta_n\rho(\AD_{[t,j]}z_n))\cap \C\neq\emptyset.$$
Choose $y_n\in \mathcal{B}(\AD_{[t,j]}z_n,\sqrt{2}\delta_n\rho(\AD_{[t,j]}z_n))\cap \C$.
It follows that
\begin{align*}
\d(y_n,\AD_{[t,j]}z)&\le \d(y_n,\AD_{[t,j]}z_n)+\d(\AD_{[t,j]}z_n,\AD_{[t,j]}z)\\
&\le \sqrt{2}\delta_n\rho(\AD_{[t,j]}z_n)+\d(\AD_{[t,j]}z_n,\AD_{[t,j]}z).
\end{align*}
{ Part of 1) of Lemma \ref{lem:conv}} shows that $\d(\AD_{[t,j]}z_n,\AD_{[t,j]}z)\ra 0$ as $n\ra \infty$. Moreover, by continuity of $\rho$, we have $\rho(\AD_{[t,j]}z_n)\ra \rho(\AD_{[t,j]}z)$. Hence $\d(y_n,\AD_{[t,j]}z)\ra 0$ as $n\ra\infty$.
Since $\C$ is closed, $\AD_{[t,j]}z\in\C$. This holds for all $t\in I$.

\emph{\bfseries Proof of (S2):} { Given any $(t,j)\in \dom z$ such that $(t,j+1)\in \dom z$, let $\set{s_n}_{n=1}^{\infty}$ be a sequence given by { part 2) of Lemma \ref{lem:conv}}. First, $\AD_{[s_n,j]}z_n\in \D_{\delta_n\rho}$ and $\AD_{[s_n,j]}z_n\rap \AD_{[t,j]}z$ imply that $\AD_{[t,j]}z\in \D$. This is similar to how we show $\AD_{[t,j]}z\in \C$ above.} Second, we prove that, given any $\eta>0$, there exists $N$ sufficiently large such that
\begin{align}\label{esteta3}
z_n(s_n,j+1)\in \G(\AD_{[t,j]}z) +\eta\B,
\end{align}
holds for all $n\ge N$. We have $z_n(s_n,j+1)\in\G_{\delta_n\rho}(\AD_{[s_n,j]}z_n)$. Let $\phi=\AD_{[s_n,j]}z_n$. For each $y\in \G_{\delta_n\rho}(\phi)$, by definition, there exists $v\in \G(\mathcal{B}(\phi,\delta_n\rho(\phi))\cap \D)$ such that $y\in v + \delta_n\rho(\phi^+_v)\B$. Similar to (\ref{eq:fdelta}), we have, for sufficiently large $n$,
$
v\in \G(\mathcal{B}(\phi,\delta_n\rho(\phi)))
\subset \G(\AD_{[t,j]}z) +\eta/2\B.
$
Moreover, $\delta_n\rho(\phi^+_v)\B\subset \eta/2\B$ for sufficiently large $n$. This is because $\phi^+_v$ for $v\in \G(\mathcal{B}(\AD_{[s_n,j]}z_n,\delta_n\rho(\AD_{[s_n,j]}z_n))\cap \D)$ eventually belongs to a compact set in $(\MD,\d)$. Hence, (\ref{esteta3}) holds. By part 2) of Lemma \ref{lem:conv}, $\lim_{n\rightarrow \infty}z_n(s_n,j+1)=z(t,j+1)$. It follows from (\ref{esteta3}) that 
\begin{align}\label{esteta4}
z(t,j+1)\in \G(\AD_{[t,j]}z) +2\eta\B,
\end{align}
for any $\eta>0$. Since $\eta$ is arbitrary and $\G(\AD_{[t,j]}z)$ is a closed set (due to outer  semicontinuity of $\G$), we have actually proved $z(t,j+1)\in \G(\AD_{[t,j]}z)$.

{ 
{ The proof for the case where $\set{z_i}_{i=1}^{\infty}$ is not locally eventually bounded is similar to that for Theorem 6.30 in \cite{goebel2012hybrid}.} The main idea is to choose $J$ to be the least of $j\in\Z^+$ such that the sequence $z_i$ restricted to $\Real\times \set{-\infty,\cdots,J-1, J}$ is not locally eventually bounded and $T$ be the least of all $t$'s for which there exists a subsequence of $\set{z_i}_{i=1}^{\infty}$ (without relabelling) and $(t_i,J)\in\dom z_i$ such that $t_i\ra t$ and $\abs{z_i(t_i)}\ra\infty$. The details are omitted.}
\end{proof}

\begin{defn}\em
Given a set $\mathcal{K}\subset\MD$, a hybrid system with finite memory $\HMD$ is said to be \emph{pre-forward complete} from $\K$ if every solution $x$ of $\HMD$ with $\AD_{[0,0]}x\in\K$ is either bounded (i.e., $\sup_{(t,j)\in\dom x}\abs{x(t,j)}<\infty$) or complete.
\end{defn}

\begin{prop}[Local uniform boundedness]\label{prop:lub}
Let $\HMD$ be well-posed and suppose that it is pre-forward complete from $\MD_{b,\lambda}$ for some $b,\lambda\in\Real_{\ge 0}$. Given any continuous function $\rho:\,\MD\ra\Real_{\ge 0}$ and a decreasing sequence $\set{\delta_i}_{i=1}^\infty$ in $(0,1)$ with $\delta_i\ra 0$ as $i\ra\infty$, then, for each $m>0$, there exists $\delta>0$, $N>0$, and $b'>0$ such that, for each solution $x$ to $(\HMD)_{\delta_i\rho}$ with $\AD_{[0,0]}x$ being $\delta$-close with some $\psi\in \MD_{b,\lambda}$, $\AD_{[t,j]}x\in \MD_{b'}$ for all $t+j<m$.
\end{prop}

\begin{proof}
Suppose the conclusion is not true. Then there exists a sequence of solutions $\set{x_i}_{i=1}^\infty$ to $(\HMD)_{\delta_i\rho}$ that is not locally eventually bounded and $\AD_{[0,0]}x_i\rap x\in \MD_{b,\lambda}$. We can extract a graphically convergent sequence of $x_i$ and use part (b) of the well-posedness theorem to show that there exists a maximal solution of $\HMD$ that blows up in finite time, which contradicts the definition of pre-forward completeness of $\HMD$.
\end{proof}

\section{Robustness of $\KL$ stability}\label{sec:robustness}

In this section, we prove that under the well-posedness condition, $\KL$ pre-asymptotic stability of hybrid systems with memory is robust in the following sense. This section is an extension of the results in Section 7.3 of \cite{goebel2012hybrid}, in particular Theorem 7.21 there, to hybrid systems with finite memory.

\begin{defn}[Robust $\mathcal{KL}$ pre-asymptotic stability]
Let $\W\subset\Real^n$ be a compact set and $\HMD$ be a hybrid system with memory defined in $\wMl$.
\begin{enumerate}[(a)]
\item The set $\W$ is \emph{robustly $\KL$ pre-asymptotically stable} for $\HMD$ if there exists a continuous function $\rho:\,\MD\ra \Real_{\ge 0}$ that is positive on
$$
\MD\backslash\W:=\set{\phi\in\MD:\,\exists(s,k)\in\dom\phi\;\text{s.t.}\;\phi(s,k)\not\in\W}$$ such that all solutions $x$ of $(\HMD)_{\rho}$ satisfy
\begin{equation}\label{eq:robust}
\abs{x(t,j)}_{\W}\le \beta(\norm{\AD_{[0,0]}x}_{\W},t+j), \quad \forall (t,j)\in\dom x.
\end{equation}

\item The set $\W$ is \emph{semi-globally practically robustly $\KL$ pre-asymptotically stable} for $\HMD$ if for every continuous function $\rho:\,\MD\ra \Real_{\ge 0}$ that is positive on $\MD\backslash\W$, the following holds: for every $\MD_{b}$ with $b\ge 0$ and every $\eps>0$, there exists $\delta\in(0,1)$ such that every solution $x$ to $(\HMD)_{\delta\rho}$ starting in $\MD_b$ satisfies
\begin{equation}\label{eq:practical}
\abs{x(t,j)}_{\W}\le \beta(\norm{\AD_{[0,0]}x}_{\W},t+j)+\eps, \quad \forall (t,j)\in\dom x.
\end{equation}
\end{enumerate}
\end{defn}

\begin{lem}\label{lem:stab1}
Let $\HMD$ be a hybrid system with memory and $\W\subset\Real^n$ be a compact set. If $\W$ is semi-globally practically robustly $\KL$ pre-asymptotically stable, then it is robustly $\KL$ pre-asymptotically stable.
\end{lem}

\begin{proof}
The proof is similar to that for Lemma 7.19 in \cite{goebel2012hybrid}. For notational convenience, let $w_2(z):=\abs{z}_{\W}$ for $z\in\Real^n$ and $w_1(\phi):=\norm{\phi}_{\W}$ for $\phi\in\MD$. Given any $\rho$ and $\beta$ in the definition of semi-global practical robust $\KL$ pre-asymptotical stability. Let $\set{r_n}_{n=-\infty}^{\infty}$ be a sequence such that $r_{n+1}\ge 4\beta(r_n,0)\ge 4r_n>0$ for all $n\in\mathbb{Z}$, where $r_n\ra 0$ as $n\ra-\infty$ and $r_n\ra\infty$ as $n\ra\infty$ and we have assumed, without loss of generality, that $\beta(s,0)\ge s$ for all $s\ge 0$. Since $\HMD$ is semi-globally practically robustly $\KL$ pre-asymptotically stable, it follows that for each $n\in\Z$, there exists $\delta_n\in(0,1)$ such that every solution $x$ to $(\HM)_{\delta_n\rho}$ with $\omega_1(\AD_{[0,0]}x)\le r_n$ satisfies
$$
\omega_2(x(t,j))\le \beta(\omega_1(\AD_{[0,0]}x),t+j)+\frac{r_{n-1}}{2},\quad \forall (t,j)\in\dom x.
$$
It follows that:\\
(i) $\omega_2(x(t,j))\le \beta(\omega_1(\AD_{[0,0]}x),0)+r_{n-1}/2\le r_{n+1}/4+r_{n-1}/2<r_{n+1}/2$ for all $(t,j)\in\dom x$, and \\
(ii) there exists $\tau_n$ such that $\omega_2(x(t,j))\le \beta(\omega_1(\AD_{[0,0]}x),t+j)+r_{n-1}/2\le \beta(r_n,t+j) + r_{n-1}/2\le r_{n-1}$ for $t+j\ge \tau_n$.

Consider the set $\mathcal{S}_n:=\set{\phi\in\MD:\,r_{n-1}\le\omega_1(\phi)\le r_n}$. This is a compact subset of $\MD$ since there exists some $b\ge 0$ such that $S_n\subset{\MD_b}$. It follows from the continuity of $\rho$ and the positiveness of $\rho$ on $\MD\backslash\W$ that $\inf_{\phi\in \mathcal{S}_n}\rho(\phi)>0$. Define $\rho':\,\MD\ra \Real_{\ge 0}$ such that $0<\rho'(\phi)\le\min\set{\delta_{n-1},\delta_n,\delta_{n+1}}\inf_{\phi\in \mathcal{S}_n}\rho(\phi)\le\min\set{\delta_{n-1},\delta_n,\delta_{n+1}}\rho(\phi)$ for $\phi\in \mathcal{S}_n$ for each $n\in\Z$. It is possible to make this $\rho'$ continuous. Indeed, one option is to define
$$
\rho'(\phi):=\eps_{n-1}+\frac{(\eps_{n}-\eps_{n-1})(\omega_1(\phi)-r_{n-1})}{r_n-r_{n-1}},\quad\phi\in \mathcal{S}_n,
$$
where $\eps_n=\min\set{\delta_{n-1},\delta_n,\delta_{n+1},\delta_{n+2}}\inf_{\phi\in \mathcal{S}_n}\rho(\phi)$. The continuity of $\rho'$ follows from that of $\omega_1(\phi)$. Moreover,
\begin{align*}
\rho'(\phi)&\le \max(\eps_{n-1},\eps_{n})\le \min\set{\delta_{n-1},\delta_n,\delta_{n+1}}\inf_{\phi\in \mathcal{S}_n}\rho(\phi)\\
&\le\min\set{\delta_{n-1},\delta_n,\delta_{n+1}}\rho(\phi),\quad \forall \phi\in \mathcal{S}_n.
\end{align*}
It can be verified that any solution $x$ of $(\HM)_{\rho'}$ with $\omega_1(\AD_{[0,0]}x)\le r_n$ satisfies: (a) $\omega_2(x(t,j))\le r_{n+1}/2$ for all $(t,j)\in\dom x$ and (b) there exists $(t,j)\in\dom x$ with $t+j\ge \tau_n$ such that $\omega_1(\AD_{[t,j]}x)\le r_{n-1}$, which in turn implies that $\omega_2(x(t,j))\le r_{n}/2$ for all $(t,j)\in\dom x$ with $t+j\ge\tau_n$. Define
$$
\beta'(r,s)=\sup\set{\omega_2(x(t,j)):\,x\in\mathcal{S}_{(\HM)_{\rho'}},\,\omega_1(\AD_{[0,0]}x)\le r,\,t+j\ge s}.
$$
In view of (a) and (b) above and following similar argument as in Lemma 7.11 of \cite{goebel2012hybrid}, it can be shown that $\beta'$ is a $\KL$-estimate for solutions of $(\HMD)_{\rho'}$. It follows that the set $\W$ is robustly $\KL$ pre-asymptotically stable for $\HMD$.
\end{proof}

\begin{lem}\label{lem:stab2}
Let $\HMD$ be a well-posed hybrid system with memory and $\W\subset\Real^n$ be a compact set. If $\W$ is $\KL$ pre-asymptotically stable, then it is semi-globally practically robustly $\KL$ pre-asymptotically stable.
\end{lem}

\begin{proof}
Fix $\eps>0$ and a continuous functional $\rho:\,\MD\ra\Real_{\ge 0}$ that is positive on $\MD\backslash\W$. Let $\beta$ be the $\KL$ function from the stability assumption on $\HMD$. For any given $b\ge 0$, there exists $m>\eps>0$ such that $\MD_{b}\subset\set{\phi\in\MD:\,\omega_1(\phi)\le m}$. For every fixed $T>0$, we show that there exists $\delta>0$ such that every solution $x$ to $(\HMD)_{\delta\rho}$ with $\omega_1(\AD_{[0,0]}x)\le m$ satisfies
\begin{equation}\label{eq:bound}
\omega_2(x(t,j))\le \beta(\omega_1(\AD_{[0,0]}x),t+j) + \eps/2,
\end{equation}
for all $(t,j)\in\dom x$ with $t+j\le 2T$.
Suppose this is not true. Then there exists a sequence $\delta_i\ra 0$ as $i\ra\infty$, a sequence $\set{x_i}_{i=1}^{\infty}$ of solutions to $(\HMD)_{\delta_i\rho}$ with $\omega_1(\AD_{[0,0]}x_i)\le m$, and a sequence of pairs $(t_i,j_i)\in\dom\phi_i$ with $t_i+j_i\le T$ such that
$
\omega_2(x_i(t_i,j_i))>\beta(\omega_1(\AD_{[0,0]}x_i),t_i+j_i) + \eps/2.
$
Without relabelling, extract a graphical convergent subsequence of $x_i$. It follows that $\AD_{[0,0]}x_i\rap \phi\in\MD_{b'}$ for some $b'>0$. Proposition \ref{prop:lub} implies that the sequence $x_i$ is locally eventually bounded. Well-posedness of $\HMD$ implies that the graphical limit $x$ of $x_i$ is a solution to $\HMD$. Without loss of generality, assume $(t_i,j_i)\ra (t^*,j^*)$ as $i\ra\infty$. It follows that $\omega_2(x(t^*,j^*))>\beta(\omega_1(\AD_{[0,0]}x),t^*+j^*)+\eps/2.
$ This violates that $\W$ is  pre-asymptotically stable for $\HMD$. Now we have shown (\ref{eq:bound}). Choose $T\ge\Delta+1$ sufficiently large such that $\beta(m,t+j)\le \eps/2$ for all $t+j\ge T$. It follows from  (\ref{eq:bound}) that
$
\omega_2(x(t,j))\le \eps
$
for all $t+j\in [T,2T]$ and $\omega_1(\AD_{[t',j']}x)\le \eps\le m$ for some $(t',j')\in\dom x$ with $t'+j'\in [2T-1,2T]$. By recursively using (\ref{eq:bound}), we can show that $\omega_2(x(t,j))\le \eps,
$
for all $(t,j)\in\dom x$ with $t+j\ge T$. This, combining with (\ref{eq:bound}), implies that the estimate required for the semi-global practical robust $\KL$ pre-asymptotical stability of $\HMD$ is satisfied.
\end{proof}

\begin{thm}[Robustness of pre-asymptotic stability]\label{thm:robustKL}
Let $\HMD$ be a well-posed hybrid system with memory and $\W\subset\Real^n$ be a compact set. If $\W$ is $\KL$ pre-asymptotically stable, then it is also robustly $\KL$ pre-asymptotically stable.
\end{thm}

\begin{proof}
It follows from Lemmas \ref{lem:stab1} and \ref{lem:stab2}.
\end{proof}

\section{Conclusions}

{\color{black} In this paper, we have proposed a framework to study hybrid systems with delays via generalized concepts of solutions. These solutions are defined on hybrid time domains and parameterized by both the real time and the number of jumps. We have proved an existence theorem for hybrid functional inclusions and a well-posedness result related to robust stability issues for hybrid systems with delays. These results provide a theoretical foundation for the development of a robust stability theory of hybrid systems with delays using generalized solutions \cite{liu2016lyapunov,liu2016invariance,liu2016razumikhin}. 

We would like to emphasize the main contributions of the paper as follows. Previous work on hybrid systems with delays has mostly been built on the uniform convergence topology to study existence of solutions and stability issues. For this reason, discontinuities caused by jumps in hybrid systems are not well handled, especially when structural properties of the solutions are concerned. Using generalized concepts of solutions, we have been able to develop several results that are not available in the literature: invariance principles \cite{liu2016invariance} and sufficient conditions for robust stability \cite{liu2016lyapunov} for hybrid systems with delays. In fact, one of the main results achieved in this paper states that pre-asymptotic stability for well-posed hybrid systems with memory is robust. We believe that this framework can effectively unify studies on discrete-time systems and continuous-time systems with delays. }

\bibliographystyle{siam}
\bibliography{hybrid}

\appendix

\newpage 

\section{Relationship among different measures of distance in $\MD$}\label{append:distance}

The following proposition shows how the uniform distance, $(\tau,\eps)$-closeness, and $(\rho,\eps)$-closeness can be used to provide an estimate of the graphical distance, or vice versa. The proof is essentially based on relations of the distance functions $\d$, $\d_\rho$, and $d(z,H)$ as established in \cite[Lemma 4.34 and Proposition 4.37]{rockafellar1998variational}.

{ \begin{prop}\label{dprop}
Consider two hybrid memory arcs $\phi,\,\psi\in\MD$. The following two statements hold:
\begin{enumerate}[(a)]
\item If $\dom\phi=\dom\psi$, then
\begin{equation}\label{uniform}
\d(\phi,\psi)\le\norm{\phi-\psi}.
\end{equation}
\item If $\widehat{\d}_{\rho}(\phi,\psi)\le \eps$ for some $\eps\ge0$ and $\rho\ge 2\bar{\rho}+m$, where $m=\max\set{d(0,\gph\phi),d(0,\gph\psi)}$ and $\bar{\rho}\ge 0$, then
\begin{equation}\label{rhoeps}
\d(\phi,\psi)\le \eps(1-e^{-\bar{\rho}})+(\bar{\rho}+m+1)e^{-\bar{\rho}}.
\end{equation}
In particular,
\begin{equation}\label{eps}
\widehat{\d}_{\rho}(\phi,\psi)\le \eps,\; \forall \rho\ge 0 \quad\Longrightarrow\quad \d(\phi,\psi)\le \eps.
\end{equation}
On the other hand, if $\d(\phi,\psi)\le\delta$ for some $\delta\ge 0$, then the graphs of $\widehat{\d}_{\rho}(\phi,\psi)\le\delta e^{\rho}$ for all $\rho\ge0$.
\item If $\widetilde{\d}_{\tau}(\phi,\psi)\le \eps$ for some $\eps\ge0$ and $\tau\ge 2\bar{\rho}+m$, where $m$ is as defined in (b) and $\bar{\rho}\ge 0$, then
\begin{equation}\label{taueps}
\d(\phi,\psi)\le \sqrt{2}\eps(1-e^{-\bar{\rho}})+(\bar{\rho}+m+1)e^{-\bar{\rho}}.
\end{equation}
In particular,
{ \begin{equation}\label{eps2}
\widetilde{\d}_{\tau}(\phi,\psi)\le \eps,\;\forall\tau\ge 0 \quad\Longrightarrow\quad \d(\phi,\psi)\le \sqrt{2}\eps.
\end{equation}}
On the other hand, if $\d(\phi,\psi)\le\delta$ for some $\delta\ge 0$, then $\widetilde{\d}_{\tau}(\phi,\psi)\le\delta e^{\rho}$, provided that
\begin{align}
\tau\ge0,\quad h=\max\Big\{\sup_{\abs{t+j}\le\tau}\abs{\phi(t,j)},\sup_{\abs{t+j}\le\tau}\abs{\phi(t,j)}\Big\},\quad \rho=\sqrt{h^2+\tau^2}, \quad\delta e^\rho<1.\label{thdr}
\end{align}
\end{enumerate}
\end{prop}

\begin{proof}
(a) Given $z\in\Real^{n+2}$, since $\gph\psi$ is closed, we can find $(t,j)\in \dom\,\psi$ such that
$d(z,\gph \psi)=\abs{z-(t,j,\psi(t,j))}.$
On the other hand, $d(z,\gph \phi)\le\abs{z-(t,j,\phi(t,j))}.$ Therefore,
\begin{align*}
&d(z,\gph \phi)-d(z,\gph \psi)\\
&\quad\le\abs{z-(t,j,\phi(t,j))}-\abs{z-(t,j,\psi(t,j))}\le\norm{\phi-\psi}.
\end{align*}
By symmetry of $\phi$ and $\psi$, we have
$\abs{d(z,\gph \phi)-d(z,\gph \psi)}\le\norm{\phi-\psi}$ for all $z\in\Real^{n+2}$. By definition, we have $\d_{\rho}(\phi,\psi)\le \norm{\phi-\psi}$ for all $\rho\ge 0$, which implies (\ref{uniform}).

(b) It follows from Lemma 4.34(b) of \cite{rockafellar1998variational} that
$$\d_{\bar{\rho}}(\phi,\psi)\le \eps.$$
Equation (\ref{rhoeps}) is a direct application of Lemma 4.41(b) of \cite{rockafellar1998variational}. The particular case (\ref{eps}) follows from letting $\rho\ra\infty$ and $\bar{\rho}\ra\infty$ on the right-hand side of (\ref{rhoeps}).

On the other hand, if $\d(\phi,\psi)\le\delta$, it follows that $\d_{\rho}(\phi,\psi)\le e^{\rho}\delta$ for all $\rho\ge 0$. Otherwise we will have $\d_{r}(\phi,\psi)> e^{r}\delta$ for some $r>0$ and then
$$
\d(\phi,\psi)\ge\int_{r}^{\infty}\d_{r}(\phi,\psi)e^{-\rho}d\rho>e^{r}\delta\int_{r}^{\infty}e^{-\rho}d\rho=\delta,
$$
which is a contradiction. By Lemma 4.34(a) of \cite{rockafellar1998variational}, the graphs of $\phi$ and $\psi$ are $(\rho,e^{\rho}\delta)$-close.

(c) The proof for (\ref{taueps}) follows from part (b) and the fact that if $\phi$ and $\psi$ are $(\tau,\eps)$-close, then their graphs are $(\tau,\sqrt{2}\eps)$-close. The particular case (\ref{eps2}) follows from letting $\tau\ra\infty$ and $\bar{\rho}\ra\infty$ on the right-hand side of (\ref{eps2}).

On the other hand, if $\d(\phi,\psi)\le\delta$, it follows from the second half of (b) that the graphs of $\phi$ and $\psi$ are $(\rho,e^{\rho}\delta)$-close for all $\rho\ge 0$. Let $(\tau,h,\delta,\rho)$ satisfy the conditions in (\ref{thdr}). For each $(t,j)\in\dom\phi$ with $\abs{t+j}\le\tau$, we have $(t,j,\phi(t,j))\in \gph\phi\cap\rho\B$. It follows that $(t,j,\phi(t,j))\in \gph\psi+e^{\rho}\delta\B$. That is, there exists $(t',j',\psi(t',j'))$ such that $\abs{(t',j',\psi(t',j'))-(t,j,\phi(t,j))}\le e^{\rho}\delta\B$. Since $e^{\rho}\delta<1$, we have $j=j'$, $\abs{t-t'}\le e^{\rho}\delta$, and $\abs{\phi(t,j)-\psi(t',j)}$. Similarly, for each $(t,j)\in\dom\psi$ with $\abs{t+j}\le\tau$, we can find $(t',j)$ such that $\abs{t-t'}\le e^{\rho}\delta$, and $\abs{\psi(t,j)-\phi(t',j)}$. This verifies that $\phi$ and $\psi$ are $(\tau,\delta e^{\rho})$-close.
\end{proof}

\section{Lemmas for Theorem \ref{thm:ext}}

{\color{black} The following two lemmas are used in the proof of Theorem \ref{thm:ext}. Lemma \ref{lem:tri} serves as a triangle inequality for estimating the distances between hybrid memory arcs using the distance $\dt$, whereas Lemma \ref{lem:cont} proves certain continuity properties of $\AD_{[t,j]}x$ with respect to $t$ for a given a hybrid arc $x$. }

\begin{lem}\label{lem:tri}
Let $\phi_i\in \MD$ ($i=1,2,3$) satisfy $\dt_{\tau_1}(\phi_1,\phi_2)\le \eps_1$ and
$\dt_{\tau_2}(\phi_2,\phi_3)\le \eps_2$ with $\tau_1\ge\eps$ and $\tau_2\ge\eps_1$. Then $\dt_{\tau}(\phi_1,\phi_3)\le \eps_1+\eps_2$, where $\tau=\min(\tau_1-\eps_2,\tau_2-\eps_1)$.
\end{lem}

\begin{proof}
Fix any $(s,k)\in\dom\phi_1$ with $\abs{s+k}\le\tau\le \tau_2-\eps_1$. Since $\phi_1$ and $\phi_2$ are $(\tau_1,\eps_1)$-close, there exists $(s',k)\in\dom\phi_2$ such that $\abs{\phi_2(s',k)-\phi_1(s,k)}\le \eps_1$ and $\abs{s'-s}\le \eps_1$. We have $\abs{s'+k}\le \abs{s+k}+\eps_1\le \tau_2$. Since $\phi_2$ and $\phi_3$ are $(\tau_2,\eps_2)$-close, there exists $(s'',k)\in\dom\phi_3$ such that
$\abs{s''-s'}\le \eps_2$ and $\abs{\phi_2(s'',k)-\phi_1(s',k)}\le \eps_1$. It follows that $\abs{s''-s}\le \eps_1+\eps_2$ and $\abs{\phi_2(s'',k)-\phi_1(s,k)}\le \eps_1+\eps_2$. Similarly, we can show that for any $(s,k)\in\dom\phi_3$, there exists $(s'',k)\in\dom\phi_1$ such that $\abs{s''-s}\le \eps_1+\eps_2$ and $\abs{\phi_2(s'',k)-\phi_1(s,k)}\le \eps_1+\eps_2$. This completes the proof.
\end{proof}

\begin{lem}\label{lem:cont}
Let $x$ be a hybrid arc with memory. For each $j\in \Z$ such that $I^j$ has nonempty interior, there exists a finite subset $\Theta\subset I^j$ such that the function $\alpha:\,I^j\ra\MD$, defined by $\alpha(t):=\AD_{[t,j]}x$, is uniformly continuous on each compact subinterval $U$ of $I^j\backslash\Theta$. 
Moreover, if $\Delta=\infty$, then $\alpha$ is uniformly continuous on each compact subinterval $U$ of $I^j$.
\end{lem}

\begin{proof}
Choose $\Theta$ to be the subset of $I^j$ consisting of all $t\in I^j$ for which there exists $(s,k)\in\Real_{\le 0}\times\Z_{\le 0}$ such that $s+k=-\Delta-1$, $(t+s,j+k)\in\dom x$ and $(t+s,j+k+1)\in\dom x$. Note that if $\Delta=\infty$, then $\Theta=\emptyset$. 
By continuity of $x(\cdot,j)$ on $I^j$, we know that $x(\cdot,j)$ is uniformly continuous on $U$. Given any $\eps>0$, we choose $0<\delta\le\eps/\sqrt{2}$ such that $\abs{t'-t''}\le \delta$ and $t',t''\in U$ implies
$\abs{x(t',j)-x(t'',j)}\le\eps/\sqrt{2}.$
Consider $\phi_1=\AD_{[t',j]}x$ and $\phi_2=\AD_{[t'',j]}x$, where $\abs{t'-t''}\le \delta$ and $t',t''\in U$. We have $\widetilde{\d}_{\tau}(\phi,\psi)\le \eps/\sqrt{2}$ for all $\tau\ge 0$.   
It follows from (\ref{eps2}) in Proposition \ref{dprop}(c) that $\d(\phi_1,\phi_2)\le\eps$. This proves that $t\mapsto\AD_{[t,j]}x$ is uniformly continuous in $t$ on $U$.  
\end{proof}

\section{Lemmas for Theorem \ref{thm:well-posedness}}\label{app:lemmas}

To prove Theorem \ref{thm:well-posedness}, the following lemma on graphical convergence of hybrid memory arcs induced by graphically convergent solutions is used. 

\begin{lem}\label{lem:conv}
Suppose a hybrid system with memory $\HMD=(\C,\F,\D,\G)$ satisfies Assumption \ref{as1}.
Let $\set{z_i}_{i=1}^{\infty}$ be a graphically convergent sequence of solutions to $(\HMD)_\rho$ with limit $z$. If $\set{z_i}_{i=1}^{\infty}$ is locally eventually bounded, then the following hold:
\begin{enumerate}
\item $\set{\AD_{[t,j]}z_i}_{i=1}^{\infty}$ graphically converges to $\AD_{[t,j]}z$ for each $t\in \text{int}(I_j)$;
\item if $(t,j)\in \dom z$ and $(t,j+1)\in \dom z$, there exists a sequence $\set{s_i}_{i=1}^{\infty}$ such that the following hold simultaneously: $(s_i,j)\in \dom(z_i)$, $(s_i,j+1)\in \dom(z_i)$, $\lim_{i\ra\infty}s_i=t$, $\lim_{i\ra\infty}z_i(s_i,j)=z(t,j)$, $\lim_{i\ra\infty}z_i(s_i,j+1)=z(t,j+1)$,  $\limg_{i\ra\infty}\AD_{[s_i,j]}z_i=\AD_{[t,j]}z,$ and $\limg_{i\ra\infty}\AD_{[s_i,j+1]}z_i=\AD_{[s_i,j+1]}z.$
\end{enumerate}
\end{lem}

To prove Lemma \ref{lem:conv}, we need the following lemma on set distances.

\begin{lem}\label{lem:setdist}
Let $\d$ denote the (integrated) set distance between closed subsets of $\Real^n$. We have the following.
\begin{enumerate}
\item $\d(\cup_{i=1}^{N}A_i,\cup_{i=1}^{N}B_i)\le \sum_{i=1}^N\d(A_i,B_i)$ for closed sets $A_i,\,B_i\subset\Real^n$ ($1\le i\le N$);
\item $\d(A+x,B+y)=\d(A+x,B+y)=\d(A,B)e^{\min(\abs{x},\abs{y})}+\abs{x-y}$ for closed sets $A,\,B\subset\Real^n$ and $x,y\in\Real^n$.
\end{enumerate}
\end{lem}

\begin{proof}
1) The proof follows from
\begin{equation}\label{eq:step1}
\d_{\rho}(\cup_{i=1}^{N}A_i,\cup_{i=1}^{N}B_i)\le \sum_{i=1}^N\d_{\rho}(A_i,B_i)
\end{equation}
for all $\rho\ge 0$. We only need to prove (\ref{eq:step1}) for $N=2$ and the rest follows from induction on $N$. Note that
{\allowdisplaybreaks \begin{align}
\d_{\rho}(A_1\cup A_2,B_1\cup B_2)&=\max_{\abs{z}\le \rho}\big\vert d(z,A_1\cup A_2)- d(z,B_1\cup B_2)\big\vert\notag\\
&=\big\vert|z_0-a|-|z_0-b|\big\vert,\label{eq:step2}
\end{align}}
where $\abs{z_0}\le \rho$, $a\in A_1\cup A_2$, and $b\in B_1\cup B_2$. We let $a_1\in A_1$, $a_2\in A_2$, $b_1\in B_1$, and $b_2\in B_2$ be such that $|z_0-a_1|=d(z,A_1)$, $|z_0-a_2|=d(z,A_2)$, $|z_0-b_1|=d(z,B_1)$, and $|z_0-b_2|=d(z,B_1)$. Furthermore, $a\in \set{a_1,a_2}$ and $b\in \set{b_1,b_2}$. Therefore, it follows from (\ref{eq:step2}) that
\begin{align*}
\d_{\rho}(A_1\cup A_2,B_1\cup B_2) &\le \big\vert|z_0-a|-|z_0-b|\big\vert\\
 &\le \big\vert|z_0-a_1|-|z_0-b_1|\big\vert + \big\vert|z_0-a_2|-|z_0-b_2|\big\vert\\
 &\le \sum_{i=1}^2\max_{\abs{z}\le \rho}\big\vert d(z,A_i)- d(z,B_i)\big\vert=\sum_{i=1}^2\d_{\rho}(A_i,B_i),
\end{align*}
which gives (\ref{eq:step1}).\\
2) Without loss of generality, suppose that $\abs{x}=\min(\abs{x},\abs{y})$. Note that
\begin{align}
\d_{\rho}(A+x,B+y)&=\max_{\abs{z}\le \rho}\big\vert d(z,A+x)- d(z,B+y)\big\vert\notag\\
&\le\max_{\abs{z}\le \rho}\big\vert d(z,A+x)- d(z,B+x)+\max_{\abs{z}\le \rho}\big\vert d(z,B+x)- d(z,B+y)\big\vert\notag\\
&\le\max_{\abs{z}\le \rho}\big\vert d(z-x,A)- d(z-x,B)\big\vert+\abs{x-y}\notag\\
&\le\max_{\abs{z'}\le \rho+\abs{x}}\big\vert d(z',A)- d(z',B)\big\vert+\abs{x-y}=\d_{\rho+\abs{x}}(A,B)+\abs{x-y}.\label{eq:step3}
\end{align}
Integrating (\ref{eq:step3}) gives
\begin{align*}
\d(A+x,B+y)&=\int_0^{\infty}\d_{\rho}(A+x,B+y)e^{-\rho}d\rho\\
&\le\int_0^{\infty}\d_{\rho+\abs{x}}(A,B)e^{-\rho-\abs{x}}e^{\abs{x}}d\rho+\abs{x-y}\\
&\le e^{\abs{x}}\int_0^{\infty}\d_{\rho'}(A,B)e^{-\rho'}d\rho'+\abs{x-y}= e^{\abs{x}}\d(A,B)+\abs{x-y}.
\end{align*}
Part 2) is proved.
\end{proof}

\begin{proof}[Proof of Lemma \ref{lem:conv}]
It follows from Example 5.19 of \cite{goebel2012hybrid} that $\dom(z)=\lim_{i\ra \infty}\dom z_i$ is a hybrid time domain. Moreover, similar to the proof of Lemma 5.28 in \cite{goebel2012hybrid}, we can show that $z$ is single-valued and locally Lipschitz. Moreover, $\set{z_i(t,j)}_{i=1}^{\infty}$ converge uniformly in $t$ to $z(t,j)$ and $\AD_{[t,j]}z$ on each compact subinterval of $\text{int}(I_j)$. Based on this, we can prove the following: for each subinterval $U$ of the interior of $I_j=\set{t:,(t,j)\in\dom z}$, the sequence of functions $\set{z_i(\cdot,j)\big\vert_{t\in U}}_{i=1}^{\infty}$, which is $z_i(\cdot,j)$ restricted on $U$, converges graphically to $z(\cdot,j)\big\vert_{t\in U}$. { This can be verified by the definition of graphical convergence.}
Based on this result, to show the graphical convergence of $\set{\AD_{[t,j]}z_i}_{i=1}^{\infty}$, we write
$
\gph \AD_{[t,j]}z_i=\cup_{k=j'}^{j} G_k^i,
$
where $j'\le j$ and
$$
G_k^i=\set{(s,k,z):\,(t+s,z)\in \gph z_i(\cdot,j+k),\,s+k\ge -\Delta_{\inf}},
$$
where
$$
\Delta_{\inf}:=\inf\Big\{\delta\ge\Delta:\,\exists(t+s,j+k)\in\dom z_i\;\text{s.t.}\;   s+k=-\delta\Big\}.
$$
We have a similar decomposition for $\gph \AD_{[t,j]}z$, with $z_i$ replaced by $z$. From these two decompositions and using Lemma \ref{lem:setdist}, we can show that that $\gph \AD_{[t,j]}z_i$ graphically converges to $\gph \AD_{[t,j]}z$ by showing that each component in the decomposition converges in the set distance $\d$.

4) First, there exist two sequences $\set{t'_i}_{i=1}^{\infty}$ and $\set{t''_i}_{i=1}^{\infty}$ such that $(t'_i,j)\in \dom(z_i)$, $(t''_i,j+1)\in \dom(z_i)$, $(t'_i,j,z_i(t'_i,j))\ra (t,j,z(t,j))$, and
$(t''_i,j+1,z_i(t''_i,j+1))\ra (t,j+1,z(t,j+1))$ as $i\ra \infty$. Let $s_i\in [t'_i,t''_i]$ be such that $(s_i,j)\in \dom(z_i)$ and $(s_i,j+1)\in \dom(z_i)$. Clearly, $s_i\ra t$ as $i\ra \infty$. We have
\begin{align*}
\abs{z_n(s_i,j)- z(t,j)}\le \abs{z_n(s_i,j)- z_i(t_i',j)}+\abs{z_i(t_i',j)- z(t,j)},
\end{align*}
where $\abs{z_i(t'_i,j)- z(t,j)}\ra 0$ as shown above and $\abs{z_i(s_i,j)- z_i(t'_i,j)}\ra 0$ since $z_i$ is locally uniformly Lipschitz for sufficiently large $i$. Similarly, $z_i(s_i,j+1)\ra z(t,j+1)$ as $n\ra\infty$. 
We can similarly decompose $\gph \AD_{[s_i,j]}z_i$, $\gph \AD_{[t,j]}z$, $\gph \AD_{[s_i,j+1]}z_i$, and $\gph \AD_{[t,j+1]}z$ as we
did above for $\gph \AD_{[t,j]}z_i$ and $\gph \AD_{[t,j]}z$. By showing the set convergence of each component in the decomposition and using Lemma \ref{lem:setdist}, we can show that  show $\limg_{i\ra\infty}\AD_{[s_i,j]}z_i=\AD_{[t,j]}z$ and $\limg_{i\ra\infty}\AD_{[s_i,j+1]}z_i=\AD_{[t,j+1]}z$.
\end{proof}

\end{document}